\documentclass[12pt]{amsart}
\usepackage{amsmath} 
\usepackage{amssymb}
\usepackage{amsfonts}

\usepackage[T1]{fontenc}

\usepackage{ndstye}

\usepackage{comment}

\newcommand{\br}{\mathbf R}

\newcommand{\Cal}{\mathcal}

\newcommand{\gs}{\gtrsim}

\newcommand{\hess}{\operatorname{Hess}}

\renewcommand{\iff}{\Leftrightarrow}
\newcommand{\im}{\operatorname{Im}}

\renewcommand{\ker}{\operatorname{Ker}}

\newcommand{\ls}{\lesssim}

\newcommand{\mn}[1]{\Vert#1\Vert}

\newcommand{\ol}{\overline}

\newcommand{\re}{\operatorname{Re}}

\newcommand{\restr}[1]{\big|_{#1}}

\newcommand{\set}[1]{\left\{\,#1\,\right\}}

\newcommand{\st}{\Sigma _2}
\newcommand{\supp}{\operatorname{\rm supp}}

\newcommand{\w}[1]{\langle #1\rangle }

\newcommand{\wf}{\operatorname{WF}}

\newcommand{\wt}{\widetilde}

\newcommand{\se}[1]{ S^#1_{1-{\varepsilon},{\varepsilon}}}

\numberwithin{equation}{section}

\begin{document}

\baselineskip 18pt 
\lineskip 2pt
\lineskiplimit 2pt

\title[Complex Limit]{Solvability and complex limit bicharacteristics}
\author[NILS DENCKER]{{\textsc Nils Dencker}}
\address{Centre for Mathematical Sciences, University of Lund, Box 118,
SE-221 00 Lund, Sweden}
\email{dencker@maths.lth.se}

\date {November 1, 2017}


\begin{abstract}
We shall study the solvability of pseudodifferential operators which are not of principal type. The operator will have complex principal symbol satisfying condition ($\Psi$) and we shall consider the limits of semibicharacteristics at the set where the principal symbol vanishes of at least second order. The convergence shall be as smooth curves, and we shall assume that the normalized complex Hamilton vector field of the principal symbol over the semicharacteristics converges to a real vector field.  
Also, we shall assume that the linearization of the real part of 
the normalized Hamilton vector field at the semibicharacteristic is tangent to and bounded on the tangent space of a Lagrangean submanifold at the semibicharacteristics, which we call a grazing Lagrangean space.
Under these conditions one can invariantly define the imaginary part of the subprincipal symbol.
If the quotient of the imaginary part of the subprincipal symbol with the norm of the Hamilton vector field switches sign from $ - $ to + on the bicharacteristics and becomes unbounded as they converge to the limit, then the operator is not solvable at the limit bicharacteristic. 
\end{abstract}

\maketitle

\thispagestyle{empty}

\section{Introduction}

We shall consider the solvability for a classical pseudodifferential
operator $P$  on a $C^\infty$ manifold $X$ which is not of principal type.   
$P\/$ is solvable at a compact set $K \subseteq X$ if the equation  
\begin{equation}\label{locsolv}
Pu = v 
\end{equation}
has a local solution $u \in \Cal D'(X)$ in a neighborhood of $K$
for any $v\in C^\infty(X)$
in a set of finite codimension. 
 
The pseudodifferential operator $P$ is classical if it has an asymptotic
expansion $p_m + p_{m-1} + \dots$ where $p_k$ is
homogeneous of degree $k$ in~$\xi$ and $p_m = {\sigma}(P)$
is the principal symbol of the operator.
$P\/$ is of principal type if the Hamilton vector field 
\begin{equation}
H_p = \sum_{j=1}^{n}\partial_{{\xi}_j} p\partial_{x_j} -
\partial_{x_j} p\partial_{\xi_j} 
\end{equation}
of the principal symbol $p = p_m$ does not have the radial
direction $\w{{\xi}, \partial_{\xi}}$ at $p^{-1}(0)$, in particular
$H_p \ne 0$ then. By homogeneity $H_p$ is well defined 
on the cosphere bundle $S^*X = \set{(x,{\xi}) \in T^*X:\ |{\xi}| =
  1}$, defined by some choice of Riemannean metric, and the principal type
condition means that $H_p $ is not degenerate on $ S^*X $.
For pseudodifferential operators of principal type, it is known
from~\cite{de:nt} and~\cite{ho:nec} that  
local solvability is equivalent to condition (${\Psi}$):
\begin{multline}\label{psicond} 
\text{$\im (ap)$
    does not change sign from $-$ to $+$}\\ 
 \text{along the oriented
    bicharacteristics of $\re (ap)$}
\end{multline}
for any $0 \ne a \in C^\infty(T^*M)$. This condition is of course trivial if the principal symbol 
is real valued. The oriented bicharacteristics are
the positive flow-outs of the Hamilton vector field $H_{\re (ap)} \ne
0$ on $\re (ap) =0$, and these are called
semibicharacteristics of $p$. 

We shall consider the case when  $P$ is not of principal type, instead the complex valued principal symbol vanishes of at least second order at the double characteristics~$\st$. We shall study necessary conditions 
for solvability when $\st$ is an involutive manifold, and since solvability is an open condition we shall assume that  $P$ satisfies condition (${\Psi}$) in the complement  of~$\st$ where it is of principal type.
Naturally, condition  (${\Psi}$) is empty on $\st$,
where instead we shall have necessary conditions on the next lower term $p_{m-1}$, called the  {\em subprincipal symbol}. The sum of the principal symbol and subprincipal symbol is called the {\em refined principal symbol}. 

Mendoza and Uhlman~\cite{MU1} studied the case when principal symbol $p$ is a product of two real symbols having transversal Hamilton vector fields at the involutive intersection~$\st$ of the characteristics. They  proved that $P$ is not solvable if the subprincipal symbol changes sign on the integral curves of these Hamilton vector fields on~$\st$, which are the limits of the bicharacteristics at~$\st$. Mendoza~\cite{Men}
generalized this to the case when the principal symbol is real and vanishes of second order at an
involutive manifold~$\st$  having an indefinite Hessian with rank equal to the
codimension of the manifold. The Hessian then gives well-defined limit
bicharacteristics over~$\st$, and $P$ is
not solvable if the subprincipal symbol changes sign on any of these
limit bicharacteristics. 
Since $\st $ is involutive, the limits of the bicharacteristics are tangent to the symplectic foliation of $ \st$, see Example~\ref{sympex}.
Thus, both  \cite{MU1} and~\cite{Men}  have constant sign of the subprincipal symbol on the limit characteristics as a necessary condition for solvability, which corresponds to condition ($P$) on the refined principal symbol.
This is natural since when the principal symbol vanishes of exactly second order one gets both directions on the limit bicharacteristics. 

These results were generalized in~\cite{de:lim} to
pseudodifferential operators with real principal symbol for which the linearization of the Hamilton vector field is tangent to and has uniform bounds on the tangent spaces of some Lagrangean manifolds at the bicharacteristics.  Then~$P$ is not solvable if condition (${\Psi}$) is not satisfied on the limit bicharacteristics, in the sense that the imaginary part of 
the subprincipal symbol switches sign from $-$ to $+$ on the
semibicharacteristics when converging to the limit
semibicharacteristic. 
The paper~\cite{de:sub} studied operators of subprincipal type, where the principal symbol  vanishes of at least second order at a
nonradial involutive manifold~$\st$  and the
subprincipal symbol is   
of principal type with Hamilton vector field tangent to~$\st$ at the
characteristics, but transversal to the symplectic foliation of~$\st$.
Then the operator was not solvable if the subprincipal 
symbol is constant on the symplectic leaves of~$\st$ after multiplication with a nonvanishing factor
and does not satisfy condition (${\Psi}$) on~$ \st$. In fact, if the principal symbol is proportional to a real symbol, then the result of~\cite{de:lim} gives nonsolvability generically when the  subprincipal symbol is not constant on the leaves.

In this paper, we shall extend the results of~\cite{de:lim} to
pseudodifferential operators with complex principal symbols. 
We shall consider the limits of
semibicharacteristics at the set $ \st $ where the principal symbol vanishes of at
least second order. The convergence shall be as smooth
curves, then the limit semibicharacteristic also is a smooth curve. 
We shall assume that the
normalized complex Hamilton vector field of the principal symbol on the semicharacteristics converges to a real
vector field on~$\st$. Then the limit semibicharacteristic are uniquely defined, and one can invariantly define the imaginary part of the subprincipal symbol.
Also, we shall assume that the linearization of the real part of 
the normalized  Hamilton vector field is tangent to and uniformly bounded on the tangent space of a Lagrangean submanifold at the semibicharacteristics, 
which we call a grazing Lagrangean space,
see~\eqref{cond02}.  We shall also assume uniform bounds on linearization of the imaginary part of the Hamilton vector field on the grazing Lagrangean space, see~\eqref{cond10}, \eqref{cond01} and Definition~\ref{lagrangedef}.

Our main result is Theorem~\ref{mainthm}, which
essentially says that under these conditions the operator is not
solvable at the limit 
semibicharacteristic if the quotient of the imaginary part of 
the subprincipal symbol with the norm of the Hamilton vector field
switches sign from $-$ to $+$ on the
semibicharacteristics and becomes unbounded as they converge to the limit
semibicharacteristic, see~\eqref{cond2}. Thus a non-homogeneous version of condition~($ \Psi $) on the refined principal symbol does not hold on the limit characteristics. This result implies the results of~\cite{de:lim}, \cite{MU1} and~\cite{Men}.

\section{Statement of results}

Let  $p$ be the principal symbol, ${\Sigma} =
p^{-1}(0)$ be the characteristics, and ${\Sigma}_2$ be the set of
double characteristics, i.e., the points on ${\Sigma}$ where $dp =
0$. Since we are going to study necessary conditions for solvability, 
we shall assume that $P$ satisfies condition (${\Psi}$) given by~\eqref{psicond} 
on $\Sigma_1 =  \Sigma\setminus \st $.
We shall study limits at~$\st $ of semibicharacteristics, and we shall assume that the normalized limit of $H_p$ 
is proportional to a  real vector field, in the sense that
\begin{equation}
|dp \wedge d\ol p| \ll |dp| \qquad \text{on $ \Gamma_j $ as $ j \to \infty $}
\end{equation}
We shall only use semibicharacteristics given by $H_{\re ap} $ such that $|\re a\nabla p| \ge c |\nabla p|$ at~${\Gamma}_j$ for some $c > 0$, where $\nabla p$ is the gradient of~$ p$. 
Let $\set{{\Gamma}_j}_{j=1}^\infty$ be a set of semibicharacteristics
of $p$ on $S^*X \bigcap \Sigma_1$ so that ${\Gamma}_j$
are bicharacteristics of $\re a_jp$ where $0 \ne  a_j \in C^\infty$ 
uniformly at~${\Gamma}_j$ and 
\begin{equation}\label{nablaest}
|\re a_j\nabla p| \ge c |\nabla p| \qquad \text{at~${\Gamma}_j$}
\end{equation}
for some fixed $c > 0$, observe that $ p = 0 $ on $ \Gamma_j $.
We shall assume that ${\Gamma}_j$
are uniformly bounded in $C^\infty$ when para\-metrized on a uniformly
bounded interval (for example  with respect
to the arc length). 
The bounds are defined with respect to some choice of 
Riemannean metric on $S^*X$, but different choices of metric will only
change the constants. In particular, we have a
uniform bound on the arc lengths:
\begin{equation}\label{hpcond0}
|{\Gamma}_j| \le C \qquad  \forall\,j
\end{equation}
In fact, we have that ${\Gamma}_j = \set{{\gamma}_j(t):\ t \in I_j}$ with
$|{\gamma}_j'(t)| \equiv 1$ and $|I_j| \le C$, then  $|{\gamma}_j^{(k)}(t)|
\le C_k$ for $t \in I_j$ and $\forall\ j$, $k \ge 1$. Let  the normalized gradient   $\wt p =
p/|\nabla p|$  and the normalized Hamilton vector field
\begin{equation*}
 H_{\wt p} = |H_p|^{-1}H_p \qquad \text{on $p^{-1}(0)\setminus \st$}
\end{equation*}
Then ${\Gamma}_j$ is uniformly bounded in $C^\infty$ if there exists positive constants $c$ and $C_k$
such that
\begin{equation}\label{hpcond}
| H_{\re a_j \wt  p}^k \nabla \re a_j \wt p| \le C_k \quad \text{and} \quad | H_{\re a_j \wt p}| \ge c \quad
 \text{at  ${\Gamma}_j$ }\quad \forall\, j, k 
\end{equation}
which implies that $|a_j| \ge c > 0$ at~$\Gamma_j $.
This means that the normalized
Hamilton vector field $H_{\re a_j \wt p}$ is
uniformly bounded in $C^\infty$ as a non-degenerate vector field over~ ${\Gamma}$, and this only depends on $ a_j\restr {\Gamma_j} $.
Observe that the semibicharacteristics
have a natural orientation given by the Hamilton vector field.  
Now the set of semibicharacteristic curves
$\set{{\Gamma}_j}_{j=1}^\infty$ is uniformly bounded in $C^\infty$ when
para\-metrized with respect to the arc length, and therefore it is a
precompact set. Thus there exists a subsequence ${\Gamma}_{j_k}$, $k
\to \infty$, that converge to a smooth 
curve ${\Gamma}$ (possibly a point), called a limit semibicharacteristic by
the following definition, which generalizes the definition in~\cite{de:lim}.

\begin{defn}
We say that a sequence of smooth curves ${\Gamma}_j$ on a smooth
manifold converges to a
smooth limit curve ${\Gamma}$ (possibly a point) if there exist
parametrizations on uniformly bounded intervals that converge in $C^\infty$. 
If $p \in C^\infty(T^*X)$, then we say that $\set{{\Gamma}_j}_{j=1}^\infty$ are a uniform family of semibicharacteristics of $p$ if~\eqref{hpcond0} and~\eqref{hpcond} hold.
A smooth curve ${\Gamma} \subset \st\
\bigcap S^*X$ is a {\em limit semibicharacteristic} of $p$ if there exists a uniform
family of semibicharacteristics of $p$ that converge to it. 
\end{defn}

Naturally, this definition is invariant under symplectic changes of coordinates, and 
the set
$\set{{\Gamma}_j}_{j=1}^\infty$ may have 
subsequences converging to several 
different limit semibicharacteristics, which could be points.
For example, if ${\Gamma}_j$ is parametrized with respect to the arc
length on intervals $I_j$ such 
that $|I_j| \to 0$, then we find that ${\Gamma}_j$ converges
to a limit curve which is a point.
Observe that if ${\Gamma}_j$ converge to a limit semibicharacteristic
${\Gamma}$, then~\eqref{hpcond0} and~\eqref{hpcond} must hold for~${\Gamma}_j$.

\begin{exe}\label{circlex}
Let ${\Gamma}_j$ be the curve parametrized by
\begin{equation*}
[0, 1]\ni  t \mapsto {\gamma}_j(t) = (t, \cos(jt)/j, \sin(jt)/j)/\sqrt{2}
\end{equation*}
Since $|{\gamma}'_j(t)| = 1$,  the curves are parametrized with respect to arc length, and we have that ${\Gamma}_j \to \Gamma = \set{(t,0,0):\ t \in \big [
0,2^{-1/2}\big]}$ in
$C^0$, but not in $C^\infty$ since $|{\gamma}_j''(t)|= j/\sqrt{2}$. If we 
parametrize ${\Gamma}_j$ with $x = jt \in [0, j]$ we find that
${\Gamma}_j$ converge to $\Gamma $ in $C^\infty$ but not on uniformly
bounded intervals.  
\end{exe} 

But we shall also need a condition on the differential of the Hamilton
vector field $H_p$ at the semibicharacteristic~$ \Gamma $ along a Lagrangean space, which will give bounds on the 
curvature of the semicharacteristics in these directions. 
If the semicharacteristics is the bicharacteristic of $ \re ap $ then 
we shall denote $ \Sigma =(\re a p)^{-1}(0) $ and $ T_w\Sigma = \ker d \re a p(w) \subset T(T^*X)$, where $ d \re a p(w) \ne 0 $ for $ w\in \Gamma $.
A section of Lagrangean spaces $L$
over a bicharacteristic ${\Gamma}$ is a map 
$$
{\Gamma} \ni w \mapsto L(w) \subset T_w(T^*X)
$$ 
such that $L(w)$ is a Lagrangean space in $T_w{\Sigma}$,
$\forall\,w \in \Gamma$.
If the section $L$ is $C^1$
then it has tangent space $TL \subset T_{L}(T_{\Gamma}(T^*X))$. 
Observe that since $L(w) \subset T_w{\Sigma}$ is Lagrangean
we find $ d \re a p(w)\restr {L(w)} = 0 $ and
$H_{\re ap}(w) \in L(w)$ when $w \in {\Gamma}$. 
Now we shall also have the condition that the {\em linearization} of $ H_{\re ap} $ at $ \Gamma $ is tangent to the Lagrangean space $ L $.

\begin{defn}\label{lagrangedef}
Let ${\Gamma}$ be a semibicharacteristic of\/ $p$, i.e., a bicharacteristic 
of $ \re (ap) $ for some $ 0 \ne a \in C^\infty $. We say that a $ C^1 $ section of
Lagrangean spaces $L$ over \/${\Gamma}$ is a section of {\em grazing
Lagrangean spaces} of\/ ${\Gamma}$ if $L \subset T_\Gamma \Sigma = \ker d \re a p\restr{\Gamma} \subset T_\Gamma (T^*X)$, and the
linearization (or first order jet) of $H_{\re a p} \subset T_\Gamma L$, the tangent space of~$ L $ at\/ $ \Gamma $.
\end{defn}

The linearization
of $H_{\re a p}(w)$ is given by the second order Taylor expansion of $\re a p$ at
$w$ and since $L(w)$ is Lagrangean we find that terms in that
expansion that vanish on $L(w)$
have Hamilton field parallel to~$L$. Thus, the condition that the
linearization of $H_{\re a p}(w)$ is in $ TL (w)$ only depends on the
restriction to $L(w)$ of the second order Taylor expansion of  $\re a p$ at
$w$. 
We find that Definition~\ref{lagrangedef} is invariant
under multiplication of~$\re a p$ by non\-vanishing real factors because $\re a  p(w)  = 0 $ and $ d \re a p(w)\restr{L(w)} = 0 $ since $ L \subset T_\Gamma\Sigma $. Thus the linearization of $H_{\re c a p}  $ is determined by $\hess \re c a p(w)\restr{L(w)}=  c \hess \re a  p(w)\restr{L(w)}$ when $c$ is real.
Thus the linearization only depends on the argument of $a_j $ at $ \Gamma_j $ so we can replace $H_{\re a p}(w)$ by 
$H_{ \re a \wt p}$ in the definition.

By Definition~\ref{lagrangedef} we find that the linearization of
$H_{\re a p}$ gives an evolution equation for the
section $L$, see Example~\ref{grazex}.
Choosing a Lagrangean subspace of $T_{w_0}{\Sigma}$
at $w_0 \in {\Gamma}$ then determines $L$ along ${\Gamma}$, so~$L$
must be smooth.
Actually, $L$ is the tangent space at ${\Gamma}$ of a smooth Lagrangean
submanifold of $(\re a p)^{-1}(0)$, see~\eqref{newtau}.

\begin{exe}\label{grazex}
Let $ p = {\tau} + i a(t,x)\xi_1 - \left(\w{A(t,x)x,x} + 2 \w{B(t,x)x,{\xi}} +
\w{C(t,x){\xi},{\xi}}\right)/2$, $(x,{\xi}) \in T^*\br^n$, where $a(t,x) \in C^\infty $ is real valued, 
$A(t,x)$, $B(t,x)$ and $C(t,x)\in C^\infty$ are
$n\times n$ matrices, such that $A(t,x) = A^t(t,x)$ and $C(t)= C^t(t,x)$ are symmetric,
and let ${\Gamma} = \set{(t,0,0,{\xi}_0):\ t \in
  I}$. Then $ H_{ \re p}= \partial_t $ at $ \Gamma $ and
$$( \re p)^{-1}(0) = \set{{\tau} = \w{\re A(t,x)x,x}/2 + \w{\re B(t,x)x,{\xi}}
  + \w{\re C(t,x){\xi},{\xi}}/2}$$ 
where $ \re F $ is the given by the real part of the elements of~$F$. The linearization of the Hamilton field $H_{ p}$ at $(t,0,0,{\xi}_0)$ is 
\begin{equation}\label{linham}
\partial_t + ia(t,0)\partial_{x_1} +  \w{A(t,0)y + B^t(t,0){\eta}, \partial_{\eta}} - \w{B(t,0)y +
        C(t,0){\eta}, \partial_y} 
\end{equation}
with $ (y,\eta) \in T (T^*\br^n) $.
Since $d \re p = d{\tau}$ at ${\Gamma}$, a $ C^1 $ section of Lagrangean spaces $L(t) \subset T_{\Gamma}\Sigma $
must be tangent to~$\Gamma$.
Thus, by choosing linear symplectic coordinates $ (y,\eta) $ we 
may obtain that
$$
L(t) = \set{(s,y,0,E(t)y): (s,y) \in \br^{n}}
$$ 
where $ E(t) \in C^1 $ is real and symmetric with $E(0) = 0$.
By applying~\eqref{linham} on $ \eta - E(t)y $, which vanishes on~$ L(t) $, we obtain that
$ L(t) $ is a grazing Lagrangean space if
\begin{equation}\label{eveqE}
 \partial_t E(t) \\
 = \re A(t,0) +  \re B(t,0)E(t) +  E(t)\re B^t(t) + E(t)\re C(t,0)E(t)
\end{equation}
Then by uniqueness we find that
$L(t)$ is constant in $t$ if and only if  $\re A(t,0) 
\equiv 0$, and then $A(t,0) =  \hess  p \restr {L(t)}$.
In general, the real part of $ \hess p \restr {L(t)} $ is given by the right hand side of~\eqref{eveqE}.
\end{exe}

\begin{exe}\label{pex}
If $ p $   is of principal type,  then one can choose $a  \ne 0$ and symplectic coordinates so that  $\re a p = \tau $ near $\Gamma = \set{(t,0,0,{\xi}_0):\ t \in I}$. Then one can take any Lagrangean plane in $\ker d\tau \restr \Gamma = T_\Gamma\Sigma $ which is tangent to $ \Gamma $.
\end{exe}

Observe that we may choose symplectic coordinates $(t,x;{\tau},{\xi})$
so that ${\tau}= \re a p$ and 
the fiber of $L(w)$ is equal to $ \set{(s,y,0, 0): (s,y) \in \br^{n}}$
at $w \in {\Gamma} 
= \set{(t,0;0,{\xi}_0): \ t \in I}$. But it is not clear that we
can do that {\em uniformly} for a family of semibicharacteristics~$ \set{\Gamma_j}$, for that we need additional conditions.
We shall assume that there exists a grazing Lagrangean space $L_j$ of
${\Gamma}_j$, $\forall\, j$, such that the
normalized Hamilton vector field $H_{\wt p}$ satisfies 
\begin{equation}\label{cond0}
 \left|d H_{\wt p}(w)\restr {L_j(w)} \right| \le C \qquad \text{for
  $w \in {\Gamma}_j$ }\quad \forall\, j
\end{equation}
This is equivalent to 
\begin{equation}\label{cond02}
\left|d H_{ p}(w)\restr {L_j(w)} \right| \le C |H_p| 
\end{equation}
for $w \in {\Gamma}_j$ since $L \subset T_\Gamma \Sigma$.
In fact, we have that $d H_{bp} = db H_p + b dH_p + dpH_b$ on~$\Sigma$.
Since the mapping ${\Gamma}_j \ni w \mapsto L_j(w)$ is determined by
the linearization of $H_{\re a_j \wt p}$ on $L_j$, thus by $d H_{\re a_j \wt
  p}(w)\restr {L_j(w)}$, 
condition~\eqref{cond0} implies that 
${\Gamma}_j \ni w \mapsto L_j(w)$ is uniformly in $C^1$, see
Example~\ref{grazex}. Observe that 
condition~\eqref{hpcond} gives~\eqref{cond0} in the direction of $T_w
{\Gamma}_j \subset L_j(w)$. Clearly condition~\eqref{cond0} is 
invariant under changes of symplectic coordinates and multiplications with
non-vanishing real factors.  
In general, we only have $d H_{\wt  p} = \Cal O(|H_p|^{-1}) $ since $d H_{ p} = \Cal O(1) $, and
by induction we find $\partial^\alpha  H_{\wt  p} = \Cal O(|H_p|^{-|\alpha|}) $, see Proposition~\ref{wtpest}

Observe that condition~\eqref{cond0} 
gives
\begin{equation}\label{cond00}
\left|d \nabla{\re a_j \wt p}(w)\restr {L_j(w)} \right| \le C \qquad \text{for
$w \in {\Gamma}_j$ }\quad \forall\, j
\end{equation}
Since $ \nabla \re a_j \wt p $ is uniformly proportional to the normal of the level surface
$(\re a_jp)^{-1}(0)$, condition~\eqref{cond00} gives a
uniform bound on the curvature of the level surface
$(\re a_jp)^{-1}(0)$ in the directions given by $L_j$ over ${\Gamma}_j$.

\begin{exe}\label{sympex}
Assume that $p(x,{\xi})$ vanishes of exactly order $k\ge 2$ at the
involutive submanifold $\st = \set{{\xi}'=0}$, ${\xi}=
({\xi}',{\xi}'') \in \br^m \times \br^{n-m}$, such that the
localization 
$${\eta} \mapsto \sum_{|{\alpha}| = k}
\partial_{{\xi}'}^{\alpha} p(x,0,{\xi}'') {\eta}^{\alpha}
$$ 
is of
principal type when ${\eta} \ne 0$. Then the semibicharacteristics of $p\/$
with $|\re a_j\nabla \wt p| \cong 1 $ satisfies~\eqref{hpcond} and \eqref{cond0}  
with $L_j = \set{{\xi}=0}$ at any point. In fact, 
$
|\partial_{{\xi}'} p(x,{\xi})|
\cong |{\xi}'|^{k-1}
$
and $\partial_{x,{\xi}''} p(x,{\xi}) = \Cal
O(|{\xi}'|^{k})$ so $H_{\wt p} = \partial_{{\xi}'}\wt p \partial_{x'} +
\Cal O(|{\xi}'|)$ and $\partial_x^{\alpha}
\nabla p = \Cal O(|{\xi}'|^{k-1})$, $ \forall\, \alpha $,
when $|{\xi}'| \ll 1$ and $|{\xi}| \cong 1$.
\end{exe}

Now for a uniform family of semibicharacteristics~$ \set{\Gamma_j}$  we shall denote
\begin{equation}\label{cond1}
  0 < \min_{\Gamma_j}|H_p| = {\kappa}_j \to 0 \qquad j \to \infty
\end{equation}
and we shall assume that 
\begin{equation}\label{cond10}
 |d p \wedge d \ol p\,| \le C {\kappa}_j^{14/3}|H_p|^2 \qquad\text{at
   ${\Gamma}_j$} 
\end{equation}
which by Leibniz' rule means that $|d \re \wt p \wedge d\im {\wt p}|
\le C {\kappa}_j^{14/3}$ on ${\Gamma}_j$. 
In fact, we have 
\begin{equation}\label{Leib}
d(ap) \wedge
d(\ol{ap}) = |a|^2d{p}\wedge d{\ol p} + 2i \im (a\ol p\, dp \wedge d\ol a ) + |p|^2 da
\wedge d\ol a
\end{equation}
where the two last terms vanish on~$\Sigma$.
This gives a measure on the complex part of~$H_p$ and gives that
$H_{\wt p}$ is proportional to a real vector field
on~${\Gamma}_j$ modulo terms that are $\Cal O({\kappa}_j^{14/3})$. 

With $ L_j$ as in \eqref{cond0} we shall assume the following condition
\begin{equation}\label{cond01}
 \left|d \restr {L_j}(d{p}\wedge d{\ol p})(w) \right|
 \le C {\kappa}_j^{4/3}|H_p|^2 \qquad \text{for $w \in {\Gamma}_j$ }\ \forall\, j
\end{equation}
where the outer differential is restricted to~$L_j$
on~${\Gamma}_j$. 
Observe that condition~\eqref{cond01} gives an estimate on the variation of the complex part of the Hamilton vector field along ~$ L $, whereas condition ~\eqref{cond0} gives an estimate on the variation of  the Hamilton vector field.
Using  \eqref{cond02}, \eqref{cond10} and~\eqref{Leib} we find that~\eqref{cond01} is equivalent to 
\begin{equation}\label{cond010}
\left |d \restr{L_j}(d\,{\re \wt p}\wedge d\,{\im \wt p})(w) \right |\le C{\kappa}_j^{4/3} \qquad \text{for $w \in {\Gamma}_j$ }\ \forall\, j
\end{equation}
In fact, the differential of the two last terms in~\eqref{Leib} vanish since $d p = 0$ on~$L_j$
and if $a = |\nabla p|^{-1}$ then $d a\restr {L_j} = \Cal O(a) $ by~\eqref{cond02}.

If $|\nabla \re\wt p\, | \cong |\nabla \wt p\,| = 1$, then we find from~\eqref{cond10}
that 
\begin{equation}\label{cond102}
| d\im {\wt p(w)}| \le C {\kappa}_j^{14/3}\qquad\text{ on $ \ker d \re \wt p(w)$}
\end{equation}
for $w \in {\Gamma}_j$. Since  $ d \restr
{L_j} d\,{\re \wt p}(w) =  \Cal O(1)  $ by~\eqref{cond0},  we find from~\eqref{cond010} that  
\begin{equation}\label{cond012}
d \restr
{L_j} d\,{\im \wt p}(w) =  \Cal O( \kappa_j^{4/3})  \qquad\text{ on $ \ker d \re \wt p(w)$}
\end{equation}
when $w \in {\Gamma}_j$. The estimates~\eqref{cond102} and~\eqref{cond012} will be needed in order to handle the imaginary part of the principal symbol as a perturbation, see Lemmas~\ref{translemma} and~\ref{transpterm}.  

Now, since the semibicharacteristics~${\Gamma}_j$ are uniform we have $ |H_{\re a_j\wt p}| \ge c $, which by~\eqref{cond10} gives
\begin{equation}\label{im}
\im (a_j\nabla \wt p) = {\beta}_j\re (a_j\nabla \wt p) + V_j   \qquad\text{at $ \gamma_j$}
\end{equation}
where $ \beta_j = \Cal O(1) $ and $|V_j| \le C {\kappa}_j^{14/3}$. The first part of the right hand side will not change the direction of $\Gamma_j$. Thus multiplying $ \wt p$ with the complex factor $ 1 -i\beta_j $ only changes the direction of the real part of the Hamilton vector field by terms that are $ \Cal O\big ({\kappa}_j^{14/3}\big) $. This only perturbs~${\Gamma}_j$ so that the distance to the original semibicharacteristic is $\Cal O({\kappa}_j^{14/3})$.
Now the derivative of the linearization of the Hamilton vector field is $\Cal O\big(|H_p|^{-2}\big) =  \Cal O\big (\kappa_j^{-2}\big )$,  see Proposition~\ref{wtpest}. Thus, the linearization is 
changed with a bounded factor and terms that are  $\Cal O\big({\kappa}_j^{8/3}\big)$. Thus, we find  from~\eqref{eveqE} that the grazing Lagrangean spaces $ L_j$ are only changed by terms that are  $\Cal O\big({\kappa}_j^{8/3}\big)$. Since $ \kappa_j \le |H_p| $ on~$\Gamma_j$ we find that conditions~\eqref{cond02}, ~\eqref{cond10} and~\eqref{cond01} are not changed. Observe that $ a_j $ is only defined on~$\Gamma_j$, but since~$\Gamma_j$ is a uniformly bounded smooth curve,  $ a_j $ can easily be uniformly extended to a neighborhood of~$\Gamma_j$.

\begin{rem}\label{semirem}
The family of uniform semibicharacteristics $ \set {\Gamma_j}_j $ satisfying condition \eqref{cond10} and the grazing Lagrangean spaces $ L_j $ of \/  $ \Gamma_j $ are invariant modulo perturbations of $ \Cal O(\kappa_j^{14/3}) $  under different choices of $a_j$ in~\eqref{hpcond}. Thus conditions~\eqref{cond02}, ~\eqref{cond10} and~\eqref{cond01} are well defined.
\end{rem}

Thus, the choice of~$a_j$ will be irrelevant when taking the limit.
Now, we shall only consider semibicharacteristics $ \Gamma_j $ with tangent vectors $ H_{\re a_j \wt p} $ so that
\begin{equation}\label{cond100}
|H_{\im a_j\wt p}| \le C{\kappa}_j^{14/3}\quad\text{and $ |a_j| > 1/C $ on~${\Gamma}_j$}
\end{equation}
which implies that $|\re \nabla a_j\wt p | \ge c > 0$ when $ \kappa_j \ll 1 $. Then the multipliers  $ a_j $ are {\em  well defined} on~$ \Gamma_j $ modulo uniformly bounded factors which have argument that are $\Cal O({\kappa}_j^{14/3})$.

The invariant subprincipal symbol~$p_s$ will be
important for the solvability of the operator near $ \st $.
For the usual Kohn-Nirenberg quantization of
pseudodifferential operators, the next lower order term is equal to 
\begin{equation}\label{subprinc}
p_{s} = p_{m-1} - \frac{1}{2i} \sum_j\partial_{{\xi}_j}\partial_{x_j} p 
\end{equation}
and for the Weyl quantization it is $p_{m-1}$. Both of these are equal to $p_{m-1}$ at the involutive manifold $ \st = \set {\xi' = 0} $ since then $\partial_\xi p \equiv 0 $ at $ \st $.

For the subprincipal symbol $p_s$ we shall have a condition that essentially means that 
condition~($\Psi$) does not hold for the subprincipal symbol.
Observe that if~\eqref{cond100} holds then the imaginary part of $a_j p_s$ is well defined 
modulo terms that are~ $\Cal O({\kappa}_j^{14/3})$.
Assuming~\eqref{cond100} we shall as in~\cite{de:lim} assume  that
\begin{equation}\label{cond2}
 \min_{\partial \Gamma_j}\int \im a_jp_s |H_p |^{-1}\, ds/ |\log
 {\kappa}_j| \to \infty \qquad j \to \infty
\end{equation}
where the integration is along the natural orientation given by
$H_{\re a_jp}$ on ~${\Gamma}_j$ starting at  $w_j 
\in \overset \circ {\Gamma}_j$. 
(Actually, it suffices that the minimum in~\eqref{cond2} is sufficiently large, depending on the norms of  the symbol of the operator.)
Since $  |H_p | \ge \kappa_j \to 0 $ on $\Gamma_j$, we find that condition~\eqref{cond2} is well defined independently of the choice of multiplier $a_j$ satisfying~\eqref{cond100}.

Observe that if~\eqref{cond2} holds then there must be
a change of sign of $\im  a_jp_s$ from $-$ to $+$ on ${\Gamma}_j$, and
\begin{equation}\label{c1}
\max_{\Gamma_j} (-1)^{\pm 1}\im  a_jp_s/|H_p| |\log {\kappa}_j| \to \infty
\qquad j \to \infty 
\end{equation} 
for both signs. 
Observe that condition~\eqref{cond2} for $ a_j $ satisfying~\eqref{cond100} is invariant under symplectic changes of
coordinates and multiplication with elliptic pseudodifferential
operators, thus under conjugation with elliptic Fourier integral
operators. In fact, multiplication only changes the subprincipal symbol 
with uniform non-vanishing factors and terms proportional to $|\nabla p\,| =
|H_p|$. By multiplying with~$a_j$ we may for simplicity assume that $a_j \equiv
1$. Then by choosing
symplectic coordinates $(t,x;{\tau},{\xi})$ near a given point $w_0 \in
{\Gamma}_j$  so that $\re p = {\alpha}{\tau}$ near $w_0$ with ${\alpha}
= |\re \nabla p| \ne 0$, 
we obtain that $\partial_{x_k}\partial_{{\xi}_k} \re p =
0$ at~${\Gamma}_j$, $ \forall\ k $, and $\partial_t\partial_{\tau} \re p
= \partial_t {\alpha} = \partial_t|\re \nabla p\,|$ at ${\Gamma}_j$ near
$w_0$. Thus, the second term in~\eqref{subprinc} only gives terms
which are either real or gives terms in
condition~\eqref{cond2} which are bounded by 
\begin{equation} 
\left| \int 
\partial_t|\re \nabla p|/|\nabla\re p|\, ds/|\log({\kappa}_j)| \right|= \Cal O(|\log(|\nabla
\re p|)|/|\log({\kappa}_j)|) = \Cal O(1)
\end{equation}
when $j \gg 1$ since $|\re \nabla p| \cong |\nabla p| \ge  {\kappa}_j \to 0$ on~${\Gamma}_j$ by~\eqref{cond100}.
Thus we obtain the following result.

\begin{rem}\label{subpr}
We may replace the subprincipal symbol $p_s$ by $p_{m-1}$
in~\eqref{cond2}, since the difference is bounded as $j \to \infty$.
\end{rem} 

One can define the {\em reduced principal symbol} as $p + p_s$, see Definition~18.1.33 
in \cite{ho:yellow}.
Then \eqref{cond2} means that a non-homogeneous version of condition~($\Psi$) does not hold for the reduced principal symbol.

\begin{exe}
If $p $ is real and vanishes of exactly order $ k \ge 2$ at an involutive manifold $ \st $, then we find that $|H_p | \cong d^{k-1}$ on $S^*X $ where $ d$ is the homogeneous distance to $\st  $.  If  $\im  p_s $
changes sign from $ - $ to +  on the semibicharacteristics and vanishes of order $\ell $ at~$ \st$, then~\eqref{cond2} holds if and only if $\ell < k - 1 $.
When $ k = 2$ this means that $\im p_s $
changes sign from $ -$ to +  on the limit bicharacteristic, as in the results of \cite{MU1} and~\cite{Men}.
\end{exe}

We shall study the microlocal solvability, which is
given by the following definition. Recall that $H^{loc}_{(s)}(X)$ is
the set of distributions that are locally in the $L^2$ Sobolev space
$H_{(s)}(X)$. 

\begin{defn}\label{microsolv}
If $K \subset S^*X$ is a compact set, then we say that $P$ is
microlocally solvable at $K$ if there exists an integer $N$ so that
for every $f \in H^{loc}_{(N)}(X)$ there exists $u \in \Cal D'(X)$ such
that $K \bigcap \wf(Pu-f) = \emptyset$. 
\end{defn}

Observe that solvability at a compact set $M \subset X$ is equivalent
to solvability at $S^*X\restr M$ by~\cite[Theorem 26.4.2]{ho:yellow},
and that solvability at a set implies solvability at a subset. Also,
by Proposition 26.4.4 in~\cite{ho:yellow} the microlocal solvability is
invariant under conjugation by elliptic Fourier integral operators and
multiplication by elliptic pseudodifferential operators.
The following is the main result of the paper.

\begin{thm}\label{mainthm}
Let $P \in {\Psi}^m_{cl}(X)$ have principal symbol ${\sigma}(P)
= p$ satisfying condition~$({\Psi})$, and subprincipal symbol $p_s$. Let
${\Gamma}_j \subset S^*X$, $ j = 1, \dots  $ be a uniform family of semibicharacteristics
of~$p$ so that~\eqref{cond02}, ~\eqref{cond10}, \eqref{cond01}
and~\eqref{cond2} 
hold for some $ a_j $ satisfying~\eqref{cond100} and grazing Lagrangean spaces~$ L_j $ of\/ $\Gamma$. Then $P$ is not  microlocally solvable at any limit
semibicharacteristics of~$\set{{\Gamma}_j}_j$. 
\end{thm}

In fact, if there exists a limit semibicharacteristic, then we can choose a subsequence of 
semibicharacteristics $ \Gamma_j $ converging to it, which gives conditions~\eqref{hpcond0}
and~\eqref{hpcond} for these $ \Gamma_j $, $ \forall\ j $. 
Observe that if the principal symbol is real, then conditions~(${\Psi}$),  \eqref{cond10} and~\eqref{cond01} are  trivially satified, and we obtain Theorem~2.9 in~\cite{de:lim}.

To prove  Theorem~\ref{mainthm} we shall use the following result. Let
$\mn{u}_{(k)}$ be the $L^2$ Sobolev norm of order $k$ for $u \in
C_0^\infty$ and $P^*$ the $L^2$ adjoint of $P$.

\begin{rem} \label{solvrem}
If $P$ is microlocally solvable at ${\Gamma}\subset S^*X$,
then Lemma 26.4.5 in~\cite{ho:yellow} gives that for any $Y \Subset
X$ such that ${\Gamma} \subset S^*Y$ there exists an integer ${\nu}$ 
and a pseudodifferential operator $A$ so that
$\wf(A) \cap {\Gamma} = \emptyset$ and
\begin{equation}\label{solvest}
 \mn {u}_{(-N)} \le C(\mn{P^*{u}}_{({\nu})} + \mn {u}_{(-N-n)} +
 \mn{Au}_{(0)}) \qquad u \in C_0^\infty(Y)
\end{equation}
where~$N$ is given by Definition~\ref{microsolv}.
\end{rem}

We shall use Remark~\ref{solvrem} to prove
Theorem~\ref{mainthm} in Section~\ref{pfsect} by constructing approximate local solutions to
$P^* u = 0$. We shall first prepare and get a microlocal normal form
for the adjoint operator, which will be done in
Section~\ref{normal}. We shall then apply $ P^* $ to an oscillatory solution, 
for which we shall solve the eikonal equation
in Section~\ref{eik} and the transport equations in Section~\ref{transp}.

\section{The normal form}\label{normal}

In the following we assume that the conditions in Theorem~\ref{mainthm} holds with some limit semibicharacteristic, observe that then~\eqref{hpcond0} and~\eqref{hpcond} hold for~${\Gamma}_j$.
We shall prepare the operator to a normal form as in~\cite{de:lim}, but since the principal symbol now is complex valued the preparation will be slightly different.
First we shall put the adjoint operator $P^*$ on a
normal form uniformly and microlocally near the semibicharacteristics
${\Gamma}_j \subset \Sigma \bigcap S^*X$ converging in $ C^\infty $ to ${\Gamma} \subset \st$. This will present some difficulties since we only have conditions at the
semibicharacteristics. By the invariance, we may multiply with an elliptic
operator so that the order of $P^*$ is $m=1$ and $P^*$ has the symbol
expansion $p + p_0 + 
\dots$, where $p$ is the principal symbol. By Remark~\ref{subpr} we may assume
that $p_0$ is the subprincipal 
symbol, and as before we shall assume~\eqref{cond100} so that $ |\re \nabla p| \cong |\nabla p| $. 
Observe that $ p = 0 $ on $ \Gamma_j $
and for the adjoint the signs in~\eqref{cond2} are
reversed, changing it to  
\begin{equation}\label{newcond2} 
\max_{\partial \Gamma_j}\int \im a_j p_0 |H_p |^{-1}\, ds/ |\log
 {\kappa}_j| \to -\infty \qquad j \to \infty
\end{equation}
where  ${\kappa}_j$ given by~\eqref{cond1}. 
Changing the starting point $w_j$ of the integration to the maximum of the integral
in~\eqref{newcond2}  only improves the estimate so we may assume that
\begin{equation}\label{c3}
 \int \im a_j p_0/|H_p|\, ds \le 0 \qquad \text{on ${\Gamma}_j$}
\end{equation}
with equality at $w_j \in {\Gamma}_j$.
Since $\nabla p_0$ and $\nabla H_p$ are bounded on $S^*X$ and $|H_p| \ge
{\kappa}_j$ on~${\Gamma}_j$, we find that
$ |H_p |$ and
$p_0/|H_p|$ only change with a fixed factor and a bounded term on an interval of
length $\ls {\kappa}_j$ on $ \Gamma_j $. Thus, we find
that integrating $\im a_j p_0/|H_p|$ over such intervals only gives bounded terms.
Therefore, by~\eqref{c1} we may assume that 
\begin{equation}\label{c4}
 |{\Gamma}_j| \gg {\kappa}_j 
\end{equation}
and that condition~\eqref{newcond2} holds on some intervals of length
$\cong {\kappa}_j$ at the endpoints of ${\Gamma}_j$.

Now we choose
\begin{equation}\label{ldef}
1 \le {\lambda}_j =
{\kappa}_j^{-1/{\varepsilon}} \iff {\kappa}_j =
{\lambda}_j^{-{\varepsilon}}
\end{equation}
for some $0 < {\varepsilon} \le 1$ to be determined later.
Then we may replace $|\log {\kappa}_j|$  with
$\log {\lambda}_j$ in~\eqref{newcond2}. 
By choosing a subsequence and renumbering, we may assume
by~\eqref{cond2} that 
\begin{equation}\label{mcond2}
 \max_{\partial \Gamma_j} \int \im a_j p_0/|H_p| \, ds \le  -j\log {\lambda}_j
\end{equation}
and that this also holds on some intervals of length $\cong {\kappa}_j$ at the
endpoints of ${\Gamma}_j$. Next, we introduce the normalized principal
and subprincipal symbols 
\begin{equation}\label{norm}
 \wt p = p/|H_p|  \qquad \text{and} \qquad \wt  p_0 = p_{0}/|H_p|
\end{equation}
Then we have that $H_{\wt p}\restr {\Gamma_j} \in C^\infty$ uniformly for the grazing Lagrangean space $ L_j$ of $ \Gamma_j$,
$|H_{\wt p}| = 1$ on ${\Gamma}_j$ and $dH_{\wt p}\restr {L_j}$ is
uniformly bounded at ${\Gamma}_j$ by~\eqref{hpcond} and~\eqref{cond0}. We find that
condition~\eqref{mcond2} becomes
\begin{equation}\label{rcond2}
 \max_{\partial \Gamma_j}\int \im a_j \wt  p_0 \, ds  \le  -j\log {\lambda}_j
\end{equation}
Observe that because of
condition~\eqref{c1} we have that 
$\partial {\Gamma}_j$ has two components since $\im a_j \wt  p_0$ has opposite
sign there, thus ${\Gamma}_j$ is a uniformly embedded curve.

In the following we shall consider a
fixed semibicharacteristic ${\Gamma}_j \subset \Sigma \bigcap S^*X$ and suppress the index $j$, so
that $ a = a_j $, ${\Gamma} = {\Gamma}_j$, $L = L_j$ and ${\kappa} =
{\lambda}^{-{\varepsilon}}  = {\kappa}_j$ for some $ \varepsilon > 0$ to be determined later. 
Observe that the preparation will be
uniform in~$j$  with $\lambda $ as para\-meter, assuming the conditions in Theorem~\ref{mainthm}. Now $H_{\re a \wt p} \in C^\infty$ uniformly on~${\Gamma}$
but not in a neighborhood. By~\eqref{hpcond} we may define the first order Taylor expansion
of~$\re a \wt p$ at~${\Gamma}$ uniformly. 
Since ${\Gamma} \in C^\infty$ uniformly, we can choose local  uniform
coordinates so that ${\Gamma} = \set{(t,0)):\ t \in I \subset
  \br}$ locally. In fact, we can take a local parametrization
${\gamma}(t)$ of ${\Gamma}$ with respect to the arc length and choose the
orthogonal space $M \subset \br^{n-1}$ 
to the tangent vector of ${\Gamma}$ at a point $w_0$ with respect to some
local Riemannean metric. Then $\br \times
M \ni (t,w) \mapsto {\gamma}(t) + w $ is uniformly bounded in $C^\infty$
with a uniformly bounded inverse near $(t_0,0)$ giving local coordinates
near~${\Gamma} = \set {(t,0): t \in I}$. We may then complete~$ t $ to a  uniform symplectic coordinate
system. Multiplying with the uniformly bounded function $a(t,0)$ we may assume 
that $a(t,0) \equiv 1$.
We can define the first order Taylor term of $\re \wt p$ at
${\Gamma}$ by 
\begin{equation}\label{qdef}
{\varrho}(t,w) = \partial_{w} \re \wt 
p(t,0)\cdot w \qquad w = (x,{\tau},{\xi})
\end{equation}
which is uniformly bounded.
This can be done locally, and by using a uniformly bounded partition of 
unity we obtain this in a fixed neighborhood of
${\Gamma}$. Going back to the original coordinates, we find that
${\varrho} \in C^\infty$ uniformly near~${\Gamma}$ and $\re \wt p - 
{\varrho} = \Cal O(d^2)$, but the
error is not uniformly bounded. Here $d$ is the homogeneous distance
to ${\Gamma}$, i.e., the distance with respect
to the homogeneous metric 
\begin{equation}\label{gdef}
dt^2 + |dx|^2 + (d{\tau}^2 + |d{\xi}|^2)/\w{({\tau},{\xi})}^2
\end{equation}
But by condition~\eqref{cond0} we find
that the second order derivatives of $\wt p$ along the Lagrangean space $L$
at ${\Gamma}$ are uniformly bounded. 
We shall use homogeneous coordinates, i.e., local coordinates which are normalized 
with respect to the homogeneous metric~\eqref{gdef}. 

By completing ${\tau} = {\varrho}$ in~\eqref{qdef} to a uniformly
bounded homogeneous symplectic  
coordinate system $(\tau, w) = (\tau, x,{\tau},{\xi})$ near ${\Gamma}$ and conjugating with the
corresponding uniformly bounded Fourier integral 
operator we may assume that
\begin{equation}\label{gammaco}
{\Gamma} = \set{(t,0;0, {\xi}_0): \ t \in I} \subset S^*\br ^n
\end{equation} 
for  $|{\xi}_0| = 1$ and some bounded interval $I \ni 0$, and that $\re \wt p \cong {\tau}$ modulo second order terms
at~${\Gamma}$. The second order terms are not uniformly bounded, but
$d\nabla  \wt p \restr {L}$ is uniformly bounded at~${\Gamma}$
by~\eqref{cond0}. Since $d\re \wt p = d \tau$ on~${\Gamma}$ we 
find that $H_{\re \wt p}\restr {\Gamma} = D_t$ and since $L\subset (d p)^{-1} (0) $ we may obtain 
that $L =  \set{(t,x;0,0)}$ at any given point at~${\Gamma}$ by 
choosing suitable linear symplectic coordinates~$ (x,\xi) $. 
We find from~\eqref{cond0} that
\begin{equation}\label{imp0}
\left|d \nabla \wt p(t,0;0, \xi_0) \restr L\right | \ls 1 \qquad t \in I
\end{equation}
Condition~\eqref{cond102} gives
\begin{equation}\label{imp1}
\left|\partial_{t,x,\xi} \im \wt p(t,0;0, \xi_0) \right | \ls \kappa^{14/3}= \lambda^{-14\varepsilon/3} \qquad t \in I
\end{equation}
and condition~\eqref{cond012} gives
\begin{equation}\label{imp2}
\left| d\, \partial_{t,x,\xi} \im \wt p (t,0;0, \xi_0) \restr L  \right | \ls \lambda^{-4\varepsilon/3}\qquad t \in I
\end{equation}
Here $a \ls b$ (and $b \gs a$) means that $a \le Cb$ for some $ C > 0$.

Let 
\begin{equation}\label{wtpdef}
 q(t,w) =  |\nabla p(t,w)| \ge {\lambda}^{-{\varepsilon}} \qquad \text{ at  
${\Gamma}$}
\end{equation}
and extend  $q$ so that it is homogeneous of degree $0$, then  $q$ is the norm of
the  {\em homogeneous gradient} of~$p$. Recall that $ \lambda \gg 1$ is a parameter that depends on the 
bicharacteristic~$ \Gamma $.
Since the symbols are homogeneous, we shall restrict them to $S^* \br^n$. There we shall choose coordinates $ (t,w) $ so that $w = 0$ on  ${\Gamma}$, and then localize in conical neighborhoods depending on the parameter~$ \lambda $. 
We have $|\nabla \wt p| \equiv 1 $ at~${\Gamma}$,
higher derivatives are not uniformly bounded but can be handled by
the using the metric 
\begin{equation}\label{geps}
 g_{\varepsilon} = (dt^2 +
|dw|^2){\lambda}^{2{\varepsilon}} \qquad w = (x,\tau,\xi)
\end{equation}
and the symbol classes $f \in S(m,g_{\varepsilon})$ defined by  
$\partial^{\alpha} f = \Cal O(m{\lambda}^{|{\alpha}|{\varepsilon}})$,
$\forall\, {\alpha}$.

\begin{prop}\label{wtpest}
If  \eqref{gammaco} and \eqref{wtpdef}  hold then $q$ is a weight for  $g_{\varepsilon}$, $q  \in S(q,
g_{\varepsilon})$ and $\wt p(t,w) \in
S({\lambda}^{-{\varepsilon}}, 
g_{\varepsilon} )$ when $|w| 
 \le c{\lambda}^{-{\varepsilon}}$  for some $c > 0$ on $S^* \br^n$ when $t \in I$.
\end{prop}

This gives  $p = q\wt p \in 
S(q{\lambda}^{-{\varepsilon}}, g_{\varepsilon})$ when $|w| 
\le c{\lambda}^{-{\varepsilon}}$. 
Observe that $b \in S^{\mu}_{1-{\varepsilon}, {\varepsilon}}$ if
and only if 
$b \in S({\lambda}^{\mu}, g_{\varepsilon})$ in homogeneous coordinates
when $|{\xi}| \cong  
{\lambda}  \gs 1$. In fact, in homogeneous coordinates $z $ this
means that $\partial_z^{\alpha} b = \Cal O(|{\xi}|^{{\mu} +
  |{\alpha}|{\varepsilon}})$. 
Therefore, we obtain by homogeneity that $ \wt p \in S^{1-{\varepsilon}}_{1-{\varepsilon},{\varepsilon}} $ and $ q^{-1
} \in S^{{\varepsilon}}_{1-{\varepsilon},{\varepsilon}}$ when $|w| \ls {\lambda}^{-\varepsilon}  \cong  |\xi|^{-\varepsilon} 
\ls 1 $.

\begin{proof}
We are going to use the previously chosen coordinates $(t,w)$  on $S^* \br^n$ so that ${\Gamma} =
\set{(t,0): \ t \in I}$. Now $\partial^2p = \Cal
O(1)$,  $q \ge {\lambda^{-\varepsilon}}$ at ${\Gamma}$ by~\eqref{wtpdef} and
\begin{equation}\label{qder}
 \partial q = \re \nabla \ol p\cdot (\partial\nabla p)/q 
 \qquad \text{when $q \ne 0$} 
\end{equation}
which is uniformly bounded.
We find that $q(s,w) \cong q(t,0)$ 
when $|s-t| + |w| \le c {\lambda^{-\varepsilon}}$ for small enough $c > 0$,
so $q$ is a 
weight for $g_{\varepsilon}$ there. This gives that 
$|p(t,w)| \ls q(t,w){\lambda^{-\varepsilon}}$, 
$|\nabla p(t,w)| = q(t,w)$ and $|\partial^\alpha p | \ls 1 \ls q{\lambda^{\varepsilon}} \ls q{\lambda^{(|\alpha| - 1)\varepsilon}}$ 
for $ |\alpha| \ge 2 $,
which gives $p \in 
S(q{\lambda^{-\varepsilon}}, g_{\varepsilon})$ when $ |w| 
\le c{\lambda^{-\varepsilon}}$ and $ t \in I$.  

We find from~\eqref{qder} that $ \partial  q = \alpha/q
$ where $\alpha \in S(q^2 \lambda^{\varepsilon}, g_\varepsilon) $  when $|w| 
\le c{\lambda^{-\varepsilon}}$
since $\nabla p \in S(q, g_{\varepsilon})$ in this domain. By
induction over the order of differentiation of $q$ we obtain from~\eqref{qder} that $q \in 
S(q,g_{\varepsilon})$ 
when $|w| \le c{\lambda^{-\varepsilon}}$, which gives the result. 
\end{proof}

As before, we take the restriction of $\wt p$ to $ |\xi| = 1 $, use local coordinates $(t,w)$  on $S^*\br ^n $ so that \eqref{gammaco} holds with $\xi_0 = 0 $ and put $Q(t,w) = {\lambda^{\varepsilon}}\wt
p(t{\lambda^{-\varepsilon}},w{\lambda^{-\varepsilon}})$ when  $t \in 
I_{\varepsilon}  = \set{t{\lambda^{\varepsilon}}:\ t \in I}$. 
Recall that $\lambda \gg 1 $ is fixed, depending on $ \Gamma$.
Then by Proposition~\ref{wtpest} we find that $Q \in C^\infty$ 
uniformly when $|w| \ls 1$ and  $t \in I_{\varepsilon}$,
$\partial_{\tau} \re Q \equiv 1$  and $|\partial_{t,x,{\xi}} \re Q| 
\equiv 0$ when $w = 0$ and $t \in I_{\varepsilon}$. Thus we find
$|\partial_{\tau} Q| \ne 0$ for 
$|w| \ls 1$ and $t \in I_{\varepsilon}$. 
By using Taylor's formula at~$ \Gamma $ we
can write $Q(t,x;{\tau},{\xi}) = {\tau} + h(t,x;{\tau},{\xi})$
when $|w| \ls 1$ and  $t \in I_{\varepsilon}$,
where $h = |\nabla\re h| = 0$ at $w = 0$.
By using the Malgrange preparation theorem, we obtain
\begin{equation*}
{\tau} = a(t,w)({\tau} + h(t,w)) + s(t,x,{\xi})
 \qquad |w| \ls 1 \quad t \in I_{\varepsilon}
\end{equation*}
where $a$ and $s \in C^\infty$ uniformly, $a \ne 0$, and on
${\Gamma}$ we have $a = 1$ and $s =
|\nabla \re s| = 0$. In fact, this can be done uniformly, first locally in $t$ and then by a
uniform partition of unity for $t \in I_\varepsilon$. This gives 
\begin{equation}\label{sdef}
a(t,w)Q(t,w) = 
{\tau} -  s(t,x,{\xi})
 \qquad |w| \ls 1 \quad t \in I_{\varepsilon}
\end{equation}
In the original coordinates, we find that 
\begin{equation*}
 {\lambda}^{\varepsilon}\wt p(t,w) =
 a^{-1}(t {\lambda}^{\varepsilon},w {\lambda}^{\varepsilon})({\tau}
 {\lambda}^{\varepsilon} 
 - s(t {\lambda}^{\varepsilon},x {\lambda}^{\varepsilon},{\xi}
 {\lambda}^{\varepsilon})) 
\end{equation*}
and thus
\begin{equation}\label{locprep}
 \wt p(t,w) = b(t,w)({\tau} - r(t,x,{\xi}))  \qquad |w| \ls
 {\lambda}^{-{\varepsilon}} \quad  t  \in I  
\end{equation}
where $0 \ne b \in S(1, g_{\varepsilon})$, $r(t,x,{\xi}) =
{\lambda^{-\varepsilon}}s(t {\lambda}^{\varepsilon},x
{\lambda}^{\varepsilon},{\xi}{\lambda}^{\varepsilon}) \in
S({\lambda^{-\varepsilon}}, 
g_{\varepsilon})$ when $  |w| \ls
{\lambda}^{-{\varepsilon}} $, and $ t  \in I $, $b = 1$ and $r = |\nabla \re r| = 0$
on ${\Gamma}$.  
By condition~\eqref{imp0} we find that 
\begin{equation}\label{cond000}
 \left|d \nabla r\restr L \right| \le C \qquad \text{at $ \Gamma $}
\end{equation}
since $r$ is constant in $\tau$. Similarly,
by conditions~\eqref{imp1} and~\eqref{imp2} we find that 
\begin{equation}\label{cond001}
|\nabla \im r | \ls \lambda^{-14\varepsilon/3} \qquad  \text{at $ \Gamma $} 
\end{equation}
and  
\begin{equation}\label{cond002}
\left|d \nabla \im r\restr L \right| \le C\lambda^{-4\varepsilon/3}  \qquad \text{at $ \Gamma $}
\end{equation}

Extending by  homogeneity, we obtain this preparation where the homogeneous distance in $(x,\xi) $ to $\Gamma$ is $\ls \lambda^{-\varepsilon} $, then \eqref{cond000}--\eqref{cond002} hold with the homogeneous gradient. Now, the symbol $ b$ is homogeneous but it is not in $S^0_{1,0} $ uniformly, instead it will have uniform bounds in a larger symbol class.
In the following, we shall denote by ${\Gamma}$ the rays in $T^*\br ^n$
that goes through the semibicharacteristic.
Recall that $\wt p = p/q$, where $q \in S(q, g_{\varepsilon})$ when $ |w| \ls \lambda^{-\varepsilon} $ and is homogeneous of degree 0. 
By homogeneity we obtain from~\eqref{locprep} that
\begin{equation*}
 b^{-1}q^{-1} p(t,x;{\tau},{\xi}) = {\tau} - r(t,x,{\xi}) 
\end{equation*}
where $b^{-1} \in S^0_{1-{\varepsilon},{\varepsilon}}$, $q^{-1} \in
S^{\varepsilon}_{1-{\varepsilon},{\varepsilon}}$ and ${\tau} - r \in
S^{1-{\varepsilon}}_{1-{\varepsilon},{\varepsilon}}$ 
when $ |\xi| \gs \lambda $ and the homogeneous distance $d(x,\xi)$ to
$(0,\xi_0)$ is less than $c |\xi|^{-\varepsilon} \ls  \lambda^{-\varepsilon}$, $c > 0$. 
In fact, in homogeneous coordinates this means that $ b^{-1} \in S(1, g_\varepsilon) $, $ q^{-1} \in S(\lambda^\varepsilon, g_\varepsilon) $ and $ r \in S(\lambda^{1-\varepsilon}, g_\varepsilon) $  when $ |\xi| \cong \lambda $.

Take a homogeneous cut-off
function ${\chi}(x,\xi) \in  S^0_{1,0} $ supported where $d(x,\xi) \ls 
\lambda^{-\varepsilon}$ so that $b \ge c_0 > 0$ in $\supp \chi$ and ${\chi} = 1$ when 
$d \le c \lambda^{-\varepsilon}$ for some $c > 0$, then we have $\chi \in  S^0_{1-{\varepsilon},{\varepsilon}} $ uniformly when $|\xi| \gs \lambda $.  We take the homogeneous symbol $B =
{\chi}b^{-1}q^{-1} \in \se {\varepsilon} $ uniformly when $ |\xi| \gs \lambda $ and we compose
the corresponding pseudodifferential operator $B \in \Psi^\varepsilon_{1-{\varepsilon},{\varepsilon}} $ with $P^*$. Since $P^* 
\in {\Psi}^1_{1,0}$ we obtain an asymptotic expansion of $BP^*$ in
$S^{1 + \varepsilon
  -j(1-{\varepsilon})}_{1-{\varepsilon},{\varepsilon}}$ for $j= 0$, 1,
$2, \dots$ when $ |\xi| \gs \lambda $. But actually the symbol is in a better class.
The principal symbol is 
\begin{equation*}
({\tau} - r(t,x,{\xi})){\chi} \in
S^{1-{\varepsilon}}_{1-{\varepsilon},{\varepsilon}}  \qquad 
   \text{for $ |\xi| \gs \lambda$}
\end{equation*}
and the calculus gives that the homogeneous term is equal to
\begin{equation}\label{sub}
\frac{i}{2}H_p({\chi}b^{-1}q^{-1}) + {\chi}b^{-1}q^{-1}p_0
\end{equation}
where  $p_0$ is the  homogeneous term of the expansion of $P^*$. As before, we shall
use homogeneous coordinates. 
Then Proposition~\ref{wtpest} gives $p = q\wt p \in
S(q{\lambda}^{-{\varepsilon}}, g_{\varepsilon}) $ when $ |\xi| \cong \lambda $
and since ${\chi}b^{-1}q^{-1} \in S(q^{-1}, g_{\varepsilon})  \subset S(\lambda^\varepsilon, g_{\varepsilon})$  when $ |\xi| \cong \lambda $, we find
that the terms in~\eqref{sub}  
are in $S({\lambda}^{\varepsilon}, g_{\varepsilon})$ when $d \ls 
{\lambda}^{-{\varepsilon}}$ and by homogeneity in  $S^{1-{\varepsilon}}_{1-{\varepsilon},{\varepsilon}}$ when  $d \ls  |\xi|^{-{\varepsilon}} \ls
{\lambda}^{-{\varepsilon}}$.
The value of $H_p$ at  ${\Gamma}$ is equal to $q\partial_t$ modulo terms with coefficients that are $ \Cal O(\lambda^{-14 \varepsilon/3}) $ by~\eqref{cond001} so the value of~\eqref{sub} is equal to
\begin{equation}\label{psub}
\frac{1}{2i}\partial_t q/q + p_0/q = \frac{D_t
|\nabla p| }{2|\nabla p|} + \frac{p_0}{|\nabla p|} \qquad \text{at ${\Gamma}$}
\end{equation}
modulo $ \Cal O(\lambda^{-8\varepsilon/3}) $.
Here $|\nabla p| = \sqrt{|\partial_x p|^2/|{\xi}|^2 + |\partial_{\xi}
  p|^2}$ is the homogeneous gradient, and the error of this approximation
is bounded by ${\lambda}^{2{\varepsilon}}$ times 
the homogeneous distance $d$ to~${\Gamma}$, since~\eqref {sub} is in $S(\lambda^\varepsilon, g_\varepsilon)$.  Observe that $p_0/|\nabla
p|$ is equal to the normalized subprincipal symbol of $P^*$ on $S^*\br^n$ given by~\eqref{norm}.
But we have to estimate the error terms in this preparation.

\begin{defn}\label{wfdef}
For $0 < \varepsilon < 1/2 $ and $R \in S^{\mu}_{{\varrho},{\delta}}$  where $\varrho + \delta \ge 1 $,
${\varrho} >{\varepsilon}$ and ${\delta} < 1 - {\varepsilon}$, we say that $S^*X \ni
(x_0,{\xi}_0) \notin \wf_{\varepsilon} (R)$ if for any $ N $ there exists $c_N > 0 $ so that
$R \in S^{-N}_{{\varrho},{\delta}}$ when
the homogeneous distance to the ray $\set{(x_0, {\varrho}\,{\xi}_0):\
{\varrho} \in \br_+}$ is less than $c_N|{\xi}|^{-{\varepsilon}}$. 

For a family of operators  $R_j \in {\Psi}^{\mu}_{{\varrho},{\delta}}$, $ j = 1,\dots$, we say that  $S^*X \ni  (x_j,{\xi}_j) \notin \wf_{\varepsilon} (R_j)$ uniformly with respect to $\lambda_j \ge 1 $, if for any $ N $ there exists $C_N > 0 $ so that $R_j \in S^{-N}_{{\varrho},{\delta}}$ uniformly in $ j$ when
the homogeneous distance to the ray $\set{(x_j, {\varrho}\, {\xi}_j):\
{\varrho} \in \br_+}$ is less than $C_N|{\xi}|^{-{\varepsilon}} \le C C_N\lambda_j^{-{\varepsilon}}$ for some $ C > 0$.
\end{defn}

By the calculus, this means that there exist $A_j\in
{\Psi}^{0}_{1-{\varepsilon},{\varepsilon}}$ so that $A_j \ge c > 0$ when the distance to the  ray through $(x_j,\xi_j) $ is less than $C_N|{\xi}|^{-{\varepsilon}} \ls \lambda_j^{-{\varepsilon}}$ such that $A_jR_j \in {\Psi}^{-N}$ uniformly. This neighborhood is in fact the points with  fixed  $g_{\varepsilon}$ distance  to the  ray through $(x_j,\xi_j) $  when $ |{\xi}| \gs {\lambda}_j$.
For example, if the homogeneous cut-off functions $\chi_j$ is equal to 1 where the homogeneous distance to the ray $\set{(x_j, {\varrho}\, {\xi}_j):\
{\varrho} \in \br_+}$ is less than $C_N \lambda_j^{-{\varepsilon}}$ then $ (x_j,{\xi}_j) \notin \wf_{\varepsilon} (1 - \chi_j)$ uniformly with respect to $\lambda_j$. 
It follows from the calculus that Definition~\ref{wfdef} is invariant under composition with classical elliptic pseudodifferential operators and 
under conjugation with elliptic homogeneous Fourier integral operators preserving the fiber, 
by the conditions on ${\varrho}$ and ${\delta}$. We also have that $\wf_{\varepsilon}(R) $ grows when $ \varepsilon$ shrinks and  $\wf_{\varepsilon}(R) \subset \wf (R)$.

Now we can use the Malgrange division theorem in
order to make the lower order terms independent on ${\tau}$ when $d \ls
 \lambda^{-{\varepsilon}} $, starting with  
the subprincipal symbol  $\wt p_0 \in
S^{{\varepsilon}}_{1-{\varepsilon},{\varepsilon}}$ of $BP^*$ given
by~\eqref{sub}. Then restricting to $ |\xi| =1 $ and
rescaling as before so that $Q_0(t,w) = 
{\lambda^{-\varepsilon}}\wt p_0(t{\lambda^{-\varepsilon}},
w{\lambda^{-\varepsilon}}) \in C^\infty$ uniformly, we obtain that  
\begin{equation*}
 Q_0(t,w) = \wt c(t,w)({\tau}
 -s(t,x,{\xi})) + \wt q_0(t,x,{\xi}) \qquad |w| \ls 1 \quad t \in
 I_{\varepsilon}
\end{equation*}
where $s$ is given by~\eqref{sdef}, and $\wt c$ and $\wt q_0$ are
uniformly in $C^\infty$. This can be 
done uniformly, first locally and then by a partition of unity for $t\in I_\varepsilon$. 
We find in the original coordinates that 
\begin{equation}\label{subpreploc}
\wt p_0(t,w) = c(t,w) (  {\tau} - r(t,x,{\xi}))
+ q_0(t,x,{\xi}) \qquad d \ls {\lambda}^{-{\varepsilon}}  \quad t \in I 
\end{equation}
where
$q_0(t,w)  = {\lambda^{\varepsilon}}\wt q_0(t{\lambda^{\varepsilon}},
w{\lambda^{\varepsilon}}) \in S(\lambda^\varepsilon, g_\varepsilon)$ and $c(t,w)  =
 {\lambda}^{2\varepsilon}\wt c(t{\lambda^{\varepsilon}},
 w{\lambda^{\varepsilon}})  \in S(\lambda^{2\varepsilon}, g_\varepsilon)$. 
By using a partition of unity, we obtain~\eqref{subpreploc}
uniformly when  the homogeneous distance to
${\Gamma}$ is  $\ls \lambda^{-\varepsilon}$.  By homogeneity we find as before that $ c$  is homogeneous of degree $ -1$ and $q_0 $  is homogeneous of degree 0, which gives $c \in 
S^{2{\varepsilon}-1}_{1-{\varepsilon},{\varepsilon}}$ and $q_0 \in
S^{\varepsilon}_{1-{\varepsilon},{\varepsilon}} $ when $|\xi| \gs \lambda$. 
Now the composition of the operators having symbols~
$c$ and~${\tau}-r$ gives error terms that are homogeneous of degree  $ -1$ and are uniformly in
$S^{3{\varepsilon}-1}_{1-{\varepsilon},{\varepsilon}}$ when $  |\xi| \gs \lambda  $. 
Thus if ${\varepsilon} < 1/3$ then by multiplication with an
pseudodifferential operator with symbol $1-c$ we can make the
subprincipal symbol independent of~${\tau}$. By iterating this
procedure we can successively make any lower order terms
independent of ${\tau}$ when 
the homogeneous distance $d$ to
${\Gamma}$ is less than~$c\lambda^{-\varepsilon}$. 
By applying a homogeneous cut-off function ${\chi}$ as before we
obtain the following result.

\begin{prop} \label{prepprop}
Assume that ~\eqref{hpcond0}, \eqref{hpcond}, \eqref{cond02}, \eqref{cond10}, \eqref{cond01} and~\eqref{cond2}
hold uniformly for~${\Gamma}_j$, $ L_j $ and $\lambda_j $ satisfying~\eqref{ldef} for some $\varepsilon > 0$.
By conjugating with uniformly bounded elliptic
homogeneous Fourier integral operators and multiplying with uniformly bounded homogeneous elliptic operators we may assume that $m=1$, $ a_j \equiv 1 $ and\/
${\Gamma}_j $ is given by~\eqref{gammaco}. If\/  $0 < {\varepsilon} < 1/3$
then for any $c > 0 $ we can obtain that $B_j  P^* = Q_j  + R_j \in
{\Psi}^{1-\varepsilon}_{1-{\varepsilon},{\varepsilon}}$ where $B_j  \in 
{\Psi}^\varepsilon_{1-{\varepsilon},{\varepsilon}}$ uniformly, 
${\Gamma}_j  \bigcap \wf_{\varepsilon} (R_j ) = \emptyset$ uniformly, and the symbol of $Q_j $ is equal to  
\begin{equation}\label{prepsymb}
 {\tau} - r(t,x,{\xi}) + q_0(t,x,{\xi}) +  r_0(t,x,{\xi}) \qquad
 \text{when $d_j(x,\xi) \le c  |\xi|^{-{\varepsilon}} \ls  \lambda_j^{-{\varepsilon}}$ and  $t \in I$}
\end{equation}
where $d_j$ is the homogeneous distance 
to ${\Gamma}_j$. 
Here $ r$ is homogeneous of degree 1 and $ q_0$ is homogenous of degree 0, 
$r \in
 S^{1-{\varepsilon}}_{1-{\varepsilon},{\varepsilon}}$, $q_0 \in 
 S^{{\varepsilon}}_{1-{\varepsilon},{\varepsilon}}$ and  $r_0 \in
 S^{3{\varepsilon} - 
 1}_{1-{\varepsilon},{\varepsilon}}$ uniformly. We also have   
 $r = | \nabla\re r| = 0$, $\nabla \im r = \Cal
O\big ({\lambda}_j^{-14{\varepsilon}/3}\big )$, $d\nabla \re r \restr {L} = \Cal O(1)$ 
and  $d\nabla \im r \restr {L} = \Cal O\big ({\lambda}_j^{-4{\varepsilon}/3}\big )$ 
on~${\Gamma}_j$. We find that $q_0$ is equal to
\begin{equation}\label{symb}
\frac{D_t |\nabla p(t,0)|}{2|\nabla p(t,0)|} +
\frac{{p_0}(t,0)}{|\nabla p(t,0)|} \qquad  \text{when $d_j(x,\xi) \le  c |\xi|^{-{\varepsilon}} \ls \lambda_j^{-{\varepsilon}}$ and $t \in I$}
\end{equation}
modulo terms that are $ \Cal
O\big ({\lambda}^{-8\varepsilon/3 } + {\lambda}^{2{\varepsilon}}d_j\big ) $ where $|\nabla p| = \sqrt{|\partial_x p|^2/|{\xi}|^2 +
  |\partial_{\xi} p|^2}$ is the homogeneous gradient of~$p$.
\end{prop}

We shall apply the operator in Proposition~\ref{prepprop} on oscillatory solutions having frequencies $ \xi $ of size $\lambda$, see Proposition~\ref{expprop}.
Observe also that the integration of the term $D_t |\nabla
p(t,0)|/2 |\nabla p(t,0)|$ in~\eqref{symb} will give terms that are
$$
\Cal O(\log (|\nabla p(t,0)|)) = \Cal O(|\log ({\lambda})| + 1)
$$ 
which do not affect condition~\eqref{cond2}.

Recall that $L$ is a smooth section of  Lagrangean spaces 
$L(w) \subset T_w{\Sigma} \subset T_w(T^*\br^{n})$, $w \in {\Gamma}$, such that
the linearization of the Hamilton vector field $H_{\re p}$ is in $TL$ at
${\Gamma}$.
Here $ \Sigma = (\re p)^{-1}(0) $ and $ T_w\Sigma = \ker d \re p(w) $ where $ d \re p(w) \ne 0 $ for $ w \in \Gamma $.
By Proposition~\ref{prepprop} we may assume that ${\Gamma}
= \set{(t,0;0,{\xi}_0):\ t \in I}$, $0 \in I$,
and we may parametrize $L(t) = L(w)$ where $w  
= (t,0,{\xi}_0)$ for $t \in I$. Now since $T^*\br^n$ is a linear space, we may
identify the fiber of $T_w(T^*\br^n)$ with $T^*\br^n$. Since $L(w) \subset
T_w\Sigma$ and $w \in {\Gamma}$  
we find that $d{\tau} = 0$ in $L(w)$. Since $L(w)$ is Lagrangean, we 
find that $t$ lines are parallel to $L(w)$. By choosing linear symplectic coordinates in
$(x,{\xi})$ we obtain that $L(0) =  \set{(s,y;0,0):\
  {(s,y)} \in \br^{n}}$, then by condition~\eqref{cond000} we find that 
$\partial_x \nabla r(0,0,{\xi}_0)$ is uniformly bounded. 
Since $d{\tau} = 0$ on $L(t)$ and $L(t)$ is  Lagrangean 
we find by continuity for small $t$ that 
\begin{equation}\label{Lex}
  L(t) = \set{(s,y;0,A(t)y):\  {(s,y)} \in
  \br^{n}} 
\end{equation}
where $A(t)$ is real, continuous and symmetric for $t \in I$ and $A(0) =
0$. Since the linearization of the Hamilton vector field $H_{\re p}$ at
${\Gamma}$ is tangent to $L$, we find that $L$ is
parallel under the flow of that linearization. 
Since $L(t)$ is Lagrangean, the evolution of $t \mapsto L(t)$  is determined by
the restriction of the second order Taylor expansion  of $r(t,w)$ to $L(t)$. 
For~\eqref{Lex} this
restriction is given by the second order Taylor expansion of
\begin{equation*}
 R(t,x) = \re r(t,x,{\xi}_0 + A(t)x)
\end{equation*}
thus $\partial_x^2R(t,0)$ is uniformly bounded by condition~\eqref{cond0}.
The linearized Hamilton vector field is
\begin{multline*}
\partial_t +  \w{\partial_x^2R(t,0)x,\partial_{\xi}} = \partial_t +
 \w{\big(\partial_x^2\re r(t,0,{\xi}_0) + \partial_x\partial_{\xi}\re r(t,0,{\xi}_0)A \\ + A\partial_{\xi}\partial_x\re r(t,0,{\xi}_0)   +
 A\partial_{\xi}^2\re r(t,0,{\xi}_0)A\big)x,\partial_{\xi}}  
\end{multline*}
Applying this on
${\xi} - A(t) x$, which vanishes identically on $L(t)$ for $t \in I$, we obtain 
that the evolution of $L(t)$ is given by
\begin{multline}\label{eveq}
 A'(t) = \partial_x^2\re r(t,0,{\xi}_0) + 
 \partial_x\partial_{\xi}\re r(t,0,{\xi}_0)A(t) \\ + A(t)\partial_{\xi}\partial_x\re r(t,0,{\xi}_0)
+  A(t)\partial_{\xi}^2\re r(t,0,{\xi}_0)A(t) 
\end{multline}
with $ A(0) =0 $. This
is locally uniquely solvable and the right-hand side is uniformly
bounded as long as $A$ is bounded. Observe that  by uniqueness, $A(t) \equiv 0$
if and only if $\partial_x^2\re r(t,0,{\xi}_0) \equiv 0$, $\forall\,t$.
But since~\eqref{eveq} is non-linear, the solution
could become unbounded if $\partial_x^2\re r \ne 0$ and
$\partial_{\xi}^2\re r \ne 0$
so that $\mn {A(s)} \to \infty$ as $s \to t_1 \in
I$. This means that the angle between $L(t) = \set{(s,y;0,A(t)y):\ (s,y) \in
  \br^{n}}$ and the vertical space $\set{(s,0;0,{\eta}): (s,{\eta}) \in
  \br^{n}}$ goes to zero, but that is only a coordinate singularity.

In general, since we identify the fiber of $T_w(T^*\br^n)$ with $T^*\br^n$ we may
define $R(t,x,{\xi})$ for each~$t$ so that
\begin{equation}\label{defR}
R(t,x,{\xi}) =\re r(t,x,{\xi}_0 + {\xi}) \quad \text{when $(0,x;0,{\xi}) \in L(t)$} 
\end{equation}
Then $R = \re r$ on $L$ and we find  that 
\begin{equation}\label{newtau}
 {\tau} - \w{R(t)z,z}/2 \in C^\infty 
\end{equation}
if $z= (x,{\xi})$ and $R(t) = \partial^2_{z} R(t,0,0)\restr L(t)$. 
Observe that we find from~\eqref{cond000} that~\eqref{newtau} 
is uniformly in $ C^\infty $ in~$ z $ and  uniformly continuous in~$ t $.
We find that $R(0) = \partial_x^2\re r(t,0, {\xi}_0)$ and in general
$R(t)$ is given by the right hand side of~\eqref{eveq}.
Now we
can complete~$t$, $  {\tau} - \w{R(t)z,z}/2$ and $(x,\xi)\restr {t=0} $ to a uniform homogeneous symplectic coordinates system so that 
${\Gamma} = \set{(t,0,{\xi}_0):\ t \in I}
$ and $ L(0) = \set{(s,y;0,0):\ {(s,y)} \in  \br^{n}}$.
In fact, $(x,{\xi})$ satisfies a linear evolution equation 
$H_{\tau}(x,\xi) = 0$ and has the same value
when $t = 0$, so $(x,\xi) = 0 $ and $H_{\tau} = \partial_t $ on~${\Gamma}$. Since this is done by integration in $ t$,
it gives a uniformly bounded linear symplectic transformation in $(x,{\xi})$
which is uniformly $ C^1 $ in $ t $. It is given by a 
uniformly bounded elliptic Fourier integral operator $ F(t) $ on $\br^{n-1}$ which  is uniformly $ C^1 $ in $ t $.  We will call this type of Fourier
integral operator a {\em $C^1$ section of Fourier integral operators  on~$\br^{n-1}$}. 
This will give uniformly bounded terms when we conjugate  $ F(t) $ with a
first order differential operator in $t$, for example
the normal form of $ P^*$ given by~\eqref{prepsymb}.
For $ t $ close to $ 0 $ the section $ F(t) $ is given by multiplication with $ e^{i\w {A(t)x,x}} $,
where $ A(t) $ solves~\eqref{eveq}. For general $ t $ we can put $ F(t) $ on this form after a linear
symplectic transformation in ~$(x,{\xi})$.
Observe that $ F(t) $ is continuous on local $ L^2 $ Sobolev spaces in~$ x $, 
uniformly in ~$ t $, since it is continuous with respect to the norm $ \mn{(1 + |x|^2 + |D_x|^2)^k u} $, $ \forall\ k $.
In fact, it suffices to check this for the generators of the group of Fourier integral operators corresponding
to linear symplectic transformations of $ (x,\xi)$, which are given by the partial Fourier transforms, linear transformations in ~$ x $ 
and multiplication with $ e^{i\w{Ax,x}} $ where~$ A $ is real and symmetric.

We find in the new coordinates
that $ p = {\tau} - r_1$, where $r_1(t,x,{\xi})$ is independent
of~${\tau}$ and satisfies $ \partial^2_{z}\re r_1(t,0,0)\restr L(t)
\equiv 0$. This follows since 
$$
p(t,x;{\tau},{\xi}) = {\tau} - \w{R(t)z,z}/2 - r_1(t,x,{\tau},{\xi}) \qquad z = (x,\xi)
$$ 
where $ \partial^2_{z}\re r_1(t,0,0)\restr L(t)
\equiv 0$. We also have that $\partial_{\tau} r_1 = -
\set{t, r_1} = -\set{t,r} \equiv 0$, which is invariant under the change of symplectic
coordinates. Similarly we find that the lower order terms
$p_j(t,x,{\xi})$ remain independent of~${\tau}$ for~$j \le 0$. Since the evolution of
$L$ is determined by the second 
order derivatives of the principal symbol along $L$ by
Example~\ref{grazex}, we find that $L(t) \equiv \set{(t,x;0,0):\ {(t,x)} \in
  \br^{n}}$ after the change of coordinates. 
Since $ L $ is a grazing Lagrangean space, the linearization of $ H_{\re p} $ at $ \Gamma $ is tangent to $ L $. Thus 
$ \partial_x \re r_1  =  \partial_x^2 \re r_1= 0 $,  $\nabla \im r_1 = \Cal O\big({\lambda}_j^{-14{\varepsilon}/3}\big)$ and
condition~\eqref{cond002} gives that 
$\partial_{t,x}\nabla \im r_1 = \Cal O\big({\lambda}_j^{-4{\varepsilon}/3}\big)$  at  $ \Gamma_j $.
Changing notation so that $r = r_1$ and
$p(t,x;{\tau},{\xi}) = {\tau} - r(t,x,{\xi})$ we obtain the following result.

\begin{prop}\label{symplprop}
By conjugating with a uniformly bounded $C^1$ section of Fourier integral operators on~$\br^{n-1}$,
we may assume that the symplectic coordinates in Proposition~\ref{prepprop} 
are chosen so that the grazing Lagrangean space $L(w) \equiv \set
{(t,x,0,0):\ (t,x) \in \br^{n}}$, $\forall\, w \in {\Gamma}$, which
gives that 
$\partial_x \re r = \partial_x^2 \re r = 0$,
$\partial_{t,x}\nabla \re r  = \Cal O(1)$,  $\nabla \im r_1 = \Cal O \big ({\lambda}_j^{-14{\varepsilon}/3} \big )$
and  $\partial_{t,x}\nabla \im r = \Cal 
O  \big({\lambda}_j^{-4{\varepsilon}/3}  \big)$ at  $ \Gamma_j $.
\end{prop}

We shall apply the adjoint $P^*$ of the
operator on the form in Proposition~\ref{prepprop} on approximate
solutions on the form 
\begin{equation}\label{udef}
 u_{\lambda}(t,x) = \exp(i{\lambda}(\w{x,{\xi}_0} + {\omega}(t,x)))
 \sum_{j=0}^{M} {\varphi}_j (t, x){\lambda}^{-j{\varrho}} 
\end{equation}
where $ |\xi_0|  = 1 $, the phase function ${\omega}(t,x) \in
S({\lambda}^{-7{\varepsilon}}, g_{3\varepsilon})$ is real valued and the amplitudes
${\varphi}_j(t,x) \in S(1,g_{\delta})$ have support where $|x| 
\ls {\lambda}^{-{\delta}}$. 
Here ${\delta} \ge {\varepsilon}$ and  
${\varrho}$ are positive constants to be determined later. 
The phase function ${\omega}(t,x)$ will be constructed
in Section~\ref{eik}, see Proposition~\ref{omegalem}.
Observe that we have assumed
that ${\varepsilon} < 1/3$ in Proposition~\ref{prepprop}, but we shall
impose further restrictions on~${\varepsilon}$ later on.
We shall assume that  ${\varepsilon} + {\delta} < 1$, then if
$p(t,x,{\xi}) \in {\Psi}^{1-{\varepsilon}}_{1-{\varepsilon},{\varepsilon}}$  
when $ |\xi| \cong \lambda $ we obtain the asymptotic expansion 
\begin{multline}\label{trevesexp}
  p(t,x,D_x)  (\exp(i{\lambda}(\w{x,{\xi}_0} + {\omega}(t,x))){\varphi}(t,x))
  \\ \sim \exp(i{\lambda}(\w{x,{\xi}_0} +{\omega}(t,x))) \sum_{{\alpha}} 
   \partial_{{\xi}}^{\alpha }
  p(t,x,{\lambda}({\xi}_0 + \partial_x{\omega}(t,x)))\Cal 
  R_{\alpha}({\omega},{\lambda},D){\varphi}(t,x)/{\alpha}!
\end{multline} 
where $\Cal
R_{\alpha}({\omega},{\lambda},D){\varphi}(t,x) =  
D_y^{\alpha}(\exp(i{\lambda} \wt
{\omega}(t,x,y)){\varphi}(t,y))\restr{y=x}$ 
with
$$
\wt {\omega}(t,x,y) = {\omega}(t,y) - {\omega}(t,x) +
(x-y)\partial_x {\omega}(t,x)
$$
and the error term is of the same size as the next term in the expansion.
See for example Theorem~3.1 in~\cite[Chapter VI]{T2}, which is for classical pseudodifferential operators, phase
functions and amplitudes, but the proof is easily adapted to the case when these depend uniformly on parameters. 
Observe that since $ |\partial_x\omega| \cong \lambda^{-4\varepsilon} \ll 1 $ the expansion only involves the values of 
$ p(t,x,\xi) $ where $ |\xi| \cong \lambda \gg 1 $. 
Using this expansion we find that if $p $ is given by \eqref{prepsymb} then
\begin{multline}\label{exp}
e^{-i{\lambda}(\w{x,{\xi}_0} + {\omega}(t,x))}p(t,x,D_{t,x}) e^{i{\lambda}(\w{x,{\xi}_0} + {\omega}(t,x))}{\varphi}(t,x) \\ 
\sim {\lambda}\big(\partial_t
{\omega}(t,x) - r(t,x,{\xi}_0 + 
 \partial_x{\omega})\big ){\varphi}(t,x) \\ + D_t {\varphi}(t,x) -
 \sum_{j}\partial_{{\xi}_j}r(t,x,{\xi}_0 + 
 \partial_x{\omega})D_{x_j}{\varphi}(t,x)  + q_0(t,x,{\xi}_0 +
\partial_x{\omega}){\varphi}(t,x)
\\ +  \sum_{jk}\partial_{{\xi}_j}\partial_{{\xi}_k}r(t,x,{\xi}_0 +
 \partial_x{\omega})(\lambda^{-1} D_{x_j}D_{x_k}{\varphi}(t,x)
 + i {\varphi}(t,x) D_{x_j}D_{x_k}{\omega}(t,x))/2  +\dots
\end{multline}
which gives an expansion in
$S({\lambda}^{1- {\varepsilon} - j(1- {\delta}- {\varepsilon})}, 
g_{{\delta}})$, 
$j \ge 0$, if ${\delta} +
{\varepsilon} < 1$ and ${\varepsilon} < 1/4$. 
In fact, since $|{\xi}| \cong {\lambda}$ every
${\xi}$ derivative on terms in
$S^{1-{\varepsilon}}_{1-{\varepsilon},{\varepsilon}}$ gives
a factor that is $\Cal O({\lambda}^{{\varepsilon}-1})$ and every $x$ derivative
of ${\varphi}$ gives a factor that is $\Cal
O({\lambda}^{{\delta}})$. A factor ${\lambda} D_x^{\alpha}{\omega}$
requires $|{\alpha}|  \ge 2$ number of ${\xi}$ derivatives of a term in
the expansion of 
$P^*$, which gives a factor that is $\Cal O({\lambda}^{1+ (-7 + 3|\alpha|)\varepsilon  - |{\alpha}|(1-{\varepsilon})}) =
\Cal O({\lambda}^{1 - 7\varepsilon - |{\alpha}|(1-4{\varepsilon})})  =
\Cal O({\lambda}^{-1 + \varepsilon })$. 
Similarly, the expansion coming from terms in $P^*$ that have symbols in
$S^{\varepsilon}_{1-{\varepsilon},{\varepsilon}}$ gives an expansion
in $S^{\varepsilon - j(1 - {\delta} -
  {\varepsilon})}_{1-{\varepsilon},{\varepsilon}}$, $j \ge 0$.
Thus, if ${\delta} + {\varepsilon} < 2/3$ and ${\varepsilon} < 1/4$
then the terms in the expansion are $\Cal O({\lambda}^{{\delta} +
2{\varepsilon} -1}) $ except the terms in~\eqref{exp}, and for the
last ones we find that 
\begin{equation}\label{exp0}
 \sum_{jk}\partial_{{\xi}_j}\partial_{{\xi}_k}r(t,x,{\xi}_0
 + \partial_x{\omega})({\lambda}^{-1}D_{x_j}D_{x_k}{\varphi} + i{\varphi}D_{x_j}D_{x_k}{\omega})
 = \Cal 
 O({\lambda}^{2{\delta} +
   {\varepsilon} -1} + {\lambda}^{3{\varepsilon} - {\delta}}) 
\end{equation}
In fact,  $\partial_{{\xi}_j}\partial_{{\xi}_k}r(t,x,{\xi}_0
+ \partial_x{\omega}) = \Cal O({\lambda}^{{\varepsilon}})$ and
 $D_{x_j}D_{x_k}{\omega} = \Cal O({\lambda}^{2{\varepsilon}}d)$
when ${\varphi} \ne 0$, since we have $D_{x_j}D_{x_k}{\omega} = 0$ when $x =
0$, and $d  = \Cal O({\lambda}^{-{\delta}})$ in $\supp {\varphi}$. 

The error terms in~\eqref{exp0} are of equal size if $2{\delta} +
{\varepsilon} -1 = 3{\varepsilon} -
{\delta}$, thus ${\delta} = (1+2
{\varepsilon})/3  \ge \varepsilon$ since  $ \varepsilon \le 1 $. Since $\delta + \varepsilon < 1 $ we obtain 
that $4\varepsilon - 1 <  3 {\varepsilon} - {\delta} 
= (7{\varepsilon} -1)/3 < 0$ if
${\varepsilon} < 1/7$ and
$1 - {\delta} - {\varepsilon} = (2 - 5{\varepsilon})/3 > 1/3$ if ${\varepsilon} < 1/5$.
Thus we obtain the following result.

\begin{prop}\label{expprop}
Assume that $p$ is given by~\eqref{prepsymb},  ${\omega}(t, x) \in
S({\lambda}^{-7{\varepsilon}}, g_{3\varepsilon})$ is real valued with
$\partial_x {\omega}(t,0) \equiv \partial_x^2 {\omega}(t,0) \equiv 0$, and
${\varphi}_j(t,x) \in S(1,g_{\delta})$ has support where $|x| 
\ls {\lambda}^{-{\delta}}$ with positive ${\delta}$ and~${\varepsilon} $.  If
${\delta} = (1+ 2{\varepsilon})/3$ and ${\varepsilon} < 1/7$,
then~\eqref{exp} has an expansion in $S({\lambda}^{1-
  {\varepsilon} - j(2 - 5{\varepsilon})/3}, 
g_{{\delta}})$, $j \ge 0$, and
is equal to 
\begin{multline}\label{exp1}
{\lambda}\big (\partial_t {\omega}(t,x) -
 r(t,x, {\xi}_0 +
 \partial_x{\omega})\big ){\varphi}(t,x) \\ +  D_t {\varphi}(t,x) -
 \sum_{j}\partial_{{\xi}_j}r(t,x,{\xi}_0 +
 \partial_x{\omega}) D_{x_j}{\varphi}(t,x) + q_0(t,x, {\xi}_0 +
 \partial_x{\omega}){\varphi}(t,x)
\end{multline}
modulo terms that are $\Cal O({\lambda}^{(7{\varepsilon} -1)/3}) = \Cal O({\lambda}^{2{\delta} +
{\varepsilon} -1})$.
\end{prop}
  
In Section~\ref{transp} we shall choose ${\varepsilon} = 1/8 $ which gives ${\delta} = 5/12$,  $ (2 - 5{\varepsilon})/3 =
11/24$ and $(7{\varepsilon} -1)/3 = -1/24$, so we may take  ${\varrho} = 1/24$ in~\eqref{udef}.

\section{The eikonal equation}\label{eik}

Making the real part of the first term in the expansion~\eqref{exp} equal to zero gives the eikonal equation
\begin{equation}\label{eic}
\partial_t {\omega} -  \re s (t,x,\partial_x {\omega}) = 0 \qquad
{\omega}(0,x) \equiv 0 
\end{equation} 
where $s(t,x,\xi) = r(t,x,{\xi}_0 + \xi ) $. The imaginary part of the first term will be treated as a perturbation.
We shall solve the eikonal equation 
approximatively after scaling, since we solve the real part it will be similar to 
the argument in~\cite{de:lim}. 
We choose coordinates $ (t,x,\xi) $ on $ S^* \br ^n$ so that  $ \Gamma $ is given by~\eqref{gammaco}. We find that  $s \in S({\lambda}^{-{\varepsilon}}, g_{\varepsilon})$ when $|x| + |\xi|  \ls \lambda^{-\varepsilon} $  by
Proposition~\ref{prepprop}, and  we may assume 
that $L(t)\equiv \set  {(t,x,0,0)}$, $\forall\, t$, by Proposition~\ref{symplprop}.
But $ s$ is also in another symbol class by the following refinement
of Proposition~\ref{prepprop}.

\begin{prop}\label{locprepprop} 
Assuming  Propositions~\ref{prepprop} and~\ref{symplprop} we have 
$$
s\in
S({\lambda}^{-7{\varepsilon}}, {\lambda}^{6{\varepsilon}}(dt^2 + |dx| ^2) +
{\lambda}^{8{\varepsilon}} |d{\xi}|^2)
$$ 
when $|x| \ls {\lambda}^{-3\varepsilon}$,  $|{\xi}| \ls
{\lambda}^{-4{\varepsilon}}$  and $ t \in I $.
\end{prop}

\begin{proof}
Since $s \in S({\lambda}^{-{\varepsilon}}, g_{\varepsilon})$ when $ |x| + |\xi| \ls \lambda^{-\varepsilon} $ by
Proposition~\ref{prepprop}, we find that 
\begin{equation}\label{sest}
| \partial_{t,x}^{\alpha}\partial_{\xi}^{\beta} s| \ls
{\lambda}^{(|{\alpha}| + |{\beta}| -1){\varepsilon}} \ls
{\lambda}^{(3|{\alpha}| + 4|{\beta}| -7){\varepsilon}} \qquad \text{}
\end{equation}
when $ |x| + |\xi| \ls \lambda^{-\varepsilon} $, if and only $|{\alpha}| + |{\beta}| -1 \le 3|{\alpha}| + 4|{\beta}|
-7$, i.e.,
\begin{equation*}
 2|{\alpha}| + 3|{\beta}| > 5
\end{equation*}
Thus, we only have to check the cases $|{\alpha}| + |{\beta}| \le 2$
and $|{\beta}| \le 1$. Since the Lagrange remainder term is in the symbol
class, we only have to check the derivatives at $x = {\xi} = 0$. Then we obtain~\eqref{sest} since $s(t,0,0) = 0$, $\partial s(t,0,0) = \Cal O({\lambda}^{^{-14\varepsilon/3}})$
by~\eqref{cond001}, $\partial_{t,x} \partial_{\xi} s(t,0,{\xi}_0) = \Cal O(1)$
and $\partial_{t,x}^2 s(t,0,0)  =
\Cal O({\lambda}^{-4\varepsilon/3})$ by~\eqref{cond000} and~\eqref{cond002}.
\end{proof}

Observe that the estimates for $ \partial \im s $ and $ \partial_{t,x}  \partial \im s $ at $\Gamma $ are better than the symbol estimates,  which will be important in the proof  of  Lemma~\ref{transpterm}.
Next,  we scale and put $(x, {\xi}) = ({\lambda}^{-3\varepsilon} y, 
{\lambda}^{-4\varepsilon}{\eta})$.  When $ |y|+ |\eta| \le c $ we find
\begin{equation}\label{fdef}
(y,\eta) \mapsto f(t,y,\eta) = {\lambda}^{7\varepsilon}s(t, {\lambda}^{-3\varepsilon} y, 
{\lambda}^{-4\varepsilon}{\eta})  \in C^\infty \qquad \text{}
\end{equation}
and $ y \mapsto \omega_0(t,y)  = {\lambda}^{7\varepsilon}\omega(t, {\lambda}^{-3\varepsilon} y )  \in C^\infty$ uniformly.  Then the eikonal equation~\eqref{eic} is
\begin{equation}\label{scaleic}
\partial_t \omega_0 - \re f(t,y, \partial_y \omega_0) \equiv 0 \qquad
\omega_0(0,y) = 0 
\end{equation} 
when $ |y| \le c $.
We can solve~\eqref{scaleic} by solving the Hamilton-Jacobi equations:
\begin{equation}\label{hamjac}
\left\{
\begin{aligned} 
&\partial_t y = -\partial_{\eta}\re f(t,y,{\eta})\\
&\partial_t{\eta} = \partial_{y}\re f(t,y,{\eta})
\end{aligned}
\right.
\end{equation}
with initial values $(y(0), {\eta}(0)) = (z,0)$. 
Since we have uniform bounds on $(y,\eta) \mapsto  f (t,y,\eta) $, we find that
\eqref{hamjac} has a uniformly 
bounded $C^\infty$ solution $(y(t),\eta(t))  $ 
if $(z,0)$ is uniformly bounded. 
By taking $ z$ derivatives of the equations, we find that $z \mapsto (y(t,z),\eta(t,z))  \in C^\infty$ uniformly. By \eqref{hamjac} we  find that $(\partial_t y,\partial_t \eta) $ is uniformly bounded, and by taking repeated $t $, $ z$ derivatives of  \eqref{hamjac} we find that  $\big (\partial_t^k\partial_z^\alpha y,\partial_t^k \partial_z^\alpha \eta\big )   =  \Cal O(\lambda^{3(k-1)\varepsilon})$.

Letting $\partial_y \omega_0(t,y(t,z))
= {\eta}(t,z) $ and  $\partial_t \omega_0(t,y(t,z)) = \re f(t,y(t,z), \eta(t,z)) = \Cal O(1)$ when $|y| \le c $, we obtain the solution $ \omega_0(t,y) \in S(1, \lambda^{6\varepsilon}dt^2 + |dy|^2)$ to~\eqref{scaleic}. (Actually, we have $\partial_t \omega_0 \in S(1, \lambda^{6\varepsilon}dt^2 + |dy|^2)$.)
Since $\nabla \re f = 0$ on ${\Gamma}$ we find by uniqueness that $y = \eta = 0 $ when $z=0 $ which  
gives $\omega_0(t,0)  \equiv \partial_t \omega_0(t,0) \equiv \partial_y \omega_0(t,0) \equiv 0$. Since $
\partial_{y,\eta} \re f(t,0,0) = \partial_y^2 \re f(t,0,0) = 0$ we find by differentiating~\eqref{scaleic} twice that
\begin{multline*}
\partial_t \partial_y^2 \omega_0(t,0) = 
\partial_y \partial_{\eta} \re f(t,0,0)\partial_y^2 \omega_0(t,0) 
+ \partial_y^2 \omega_0(t,0)  \partial_{\eta}\partial_y \re f(t,0,0) \\ 
+   \partial_{y}^2 
\omega_0(t,0) \partial_{\eta}^2 \re f(t,0,0) \partial_{y}^2 \omega_0(t,0) 
\end{multline*}
Since $\partial_x^2{\omega}(0,x)\equiv 0$ 
we find by uniqueness that $\partial_x^2  
{\omega}(t,0) \equiv 0$. 

In the original coordinates we find that
that if $x(0) = \Cal O({\lambda}^{-3\varepsilon})$ and ${\xi}(0) = 0$
then  $x(t,x_0)=
\Cal O({\lambda}^{-3\varepsilon})$ and ${\xi}(t,x_0)  = \Cal
O({\lambda}^{-4\varepsilon})$ for any ~$t\in I$. 
The scaling also gives that
\begin{equation}\label{omegadef}
\omega(t,x) = \lambda^{-7\varepsilon} \omega_0(t, \lambda^{3\varepsilon}x) 
\in S({\lambda}^{-7{\varepsilon}}, g_{3{\varepsilon}}) \qquad |x| \ls \lambda^{-3\varepsilon}
\end{equation}
and we have $ \omega(t,0)\equiv \partial_x \omega(t,0) \equiv \partial_x^2 \omega(t,0) \equiv 0 $. (Actually, $\partial_t \omega(t,x)  \in S({\lambda}^{-7{\varepsilon}}, g_{3{\varepsilon}})$ when $|x| \ls \lambda^{-3\varepsilon} $.)
By the symbol estimates, we find $ \partial \omega(t,x) = \Cal O(\lambda^{-4\varepsilon})$ when  $|x| \ls \lambda^{-3\varepsilon} $.
Thus, we obtain the following result.

\begin{prop}\label{omegalem}
Let\/ $0 < {\varepsilon} < 1/3$, and assume that
Propositions~\ref{prepprop} and ~\ref{symplprop} hold.
Then there exists a real ${\omega}(t,x) \in
S({\lambda}^{-7{\varepsilon}}, 
g_{3{\varepsilon}})$ satisfying $\partial_t {\omega} =
\re r(t,x,{\xi}_0 + \partial_x {\omega})$ when $|x| \ls
{\lambda}^{-3{\varepsilon}}$  and $t \in I$ such that ${\omega}(t,0)
\equiv  \partial_x{\omega}(t,0) \equiv \partial_x^2{\omega}(t,0) \equiv 0$.
If $ {3\varepsilon} \le \delta \le  {4\varepsilon}  $ we  find that the values of\/
$(t,x;  {\lambda}\partial_{t}{\omega}(t,x), {\lambda}({\xi}_0 +
\partial_{x}{\omega}(t,x)))$ have homogeneous distance  $\ls {\lambda}^{-\delta} $ to the rays through\/ ${\Gamma}$
when $|x| \ls {\lambda}^{-\delta}$ and $t \in I$.
\end{prop}

\section{The transport equations}\label{transp}

The next term in~\eqref{exp} is the transport equation, which by homogeneity is
equal to
\begin{equation}\label{trans}
D_p {\varphi}  +
 {q_0}{\varphi} + i r_0 \varphi = 0 \qquad \text{at ${\Gamma} = \set{(t,0;0,\xi_0): \  t \in I}$}
\end{equation}
where $D_p = D_t - \sum_{j}\partial_{{\xi}_j} r(t,x,{\xi}_0 +\partial_x
{\omega(t,x)}) D_{x_j}$
\begin{equation}\label{key}
r_0(t,x) = \lambda \im r(t,x,\xi_0 + \partial_x \omega(t,x)) 
\end{equation}
and 
\begin{equation}\label{q0def}
 q_0(t) \cong D_t
  |\nabla p(t,0,{\xi}_0)|/2|\nabla p(t,0,{\xi}_0)| + {p_0}(t,0,{\xi}_0)/|\nabla
  p(t,0,{\xi}_0)| = \Cal O({\lambda}^{\varepsilon}) 
\end{equation}
modulo $\Cal O({\lambda}^{-8{\varepsilon}/3} + {\lambda}^{2{\varepsilon}}|x|)$ when $|x| \ls
{\lambda}^{-\varepsilon}$ by~\eqref{symb}. Here the real valued $ {\omega}(t,x) \in
S({\lambda}^{-7{\varepsilon}}, 
g_{3{\varepsilon}}) $ is given
by Proposition~\ref{omegalem}. Since the transport equation is given by a complex vector field, the treatment is different to the one in~\cite{de:lim}. But essentially we shall treat the complex part of the transport equation as a perturbation.

\begin{lem}\label{translemma}
If $3{\varepsilon} \le  {\delta} \le 7{\varepsilon}/2$ then we have that  
\begin{equation*}
 D_p = D_t + \sum_{j}^{} \w{a_j(t)\cdot x} D_{x_j} + R(t,x,D) 
\end{equation*}
where $ a_j(t) \in C^\infty(\br, \br^{n-1})$ uniformly, $ \forall\, j$, and $R(t,x,D)$ is a first order
differential operator in $x$ with 
coefficients that are $\Cal O({\lambda}^{3{\varepsilon}-2{\delta}})$ when $|x|
\ls {\lambda}^{-{\delta}}$.
\end{lem}

\begin{proof}
As before we shall use the translation $s(t,x,\xi) = r(t,x,\xi_0 + \xi) $, then
\begin{equation}\label{rsymb}
 s(t,x, \xi) \in S(\lambda^{-\varepsilon}, g_\varepsilon) \bigcap S({\lambda}^{-7{\varepsilon}}, {\lambda}^{6{\varepsilon}}(dt^2 + |dx|^2) +
{\lambda}^{8{\varepsilon}} |d{\xi}|^2)  
\end{equation}
 when $|x| \ls {\lambda}^{-3\varepsilon}$,  $|{\xi}| \ls
{\lambda}^{-4{\varepsilon}}$  and $ t \in I $ by Proposition~\ref {omegalem}.
Since $\partial_x^2{\omega}(t,0) \equiv 0$ we find from Taylor's
formula that $a_j(t) = 
 - \partial_x \partial_{{\xi}_j} \re s(t,0,0) $ 
 which is uniformly bounded by~\eqref{cond000}.
The coefficients of the error term $R$ are given by $\partial_{\xi}\im
s$ and the second order Lagrange remainder term of the coefficients of
$\partial_{\xi}\re s$. By Propositions~\ref{prepprop},  \ref{symplprop}
and~\ref{omegalem} we find from Taylor's formula that  
\begin{multline*}
 \partial_{\xi} \im s(t,x, \partial_x {\omega}(t,x)) = \partial_{\xi} \im s(t,0,0) + \partial_x\partial_{\xi} \im s(t,0,0) x \\ + \partial_\xi ^2 \im s(t,0,0) \partial_x {\omega}(t,x) + \Cal O\big (\lambda^{2\varepsilon} (|x|^2 + \lambda^{4\varepsilon} |x|^4) \big) \\ =
  \Cal O \big({\lambda}^{-4{\varepsilon}} +
  {\lambda}^{-{\varepsilon}}|x| +  {\lambda}^{3{\varepsilon}}|x|^2 + \lambda^{-6\varepsilon} \big) =
  \Cal O({\lambda}^{3{\varepsilon}-2{\delta}})
\end{multline*}
when $|x| \ls {\lambda}^{-{\delta}}$ since $3 \varepsilon \le  \delta \le 7\varepsilon/2 $. In fact, $\partial_{\xi} \im
s = \Cal O({\lambda}^{-14{\varepsilon}/3})$ and 
$\partial_x\partial_{\xi} \im s = \Cal
O({\lambda}^{-4{\varepsilon}/3})$ at~$\Gamma  $, $\partial^2_\xi s = \Cal
O({\lambda}^{{\varepsilon}})$, $\partial^3 s = \Cal
O({\lambda}^{2{\varepsilon}})$ and $\partial_x {\omega}(t,x) = \Cal
O({\lambda}^{2{\varepsilon}}|x|^2)  = \Cal O(\lambda^{-4\varepsilon})$
when $|x| \ls {\lambda}^{-{\delta}}$ since  
${\delta} \ge 3{\varepsilon}$.
Similarly
we find that the second order Lagrange remainder term of the coefficients of
$\partial_{\xi}\re s$ are
$ 
\Cal O({\lambda}^{2{\varepsilon}}(|x|^2 +  {\lambda}^{4{\varepsilon}}|x|^4) )
= \Cal O({\lambda}^{2{\varepsilon}-2{\delta}})
$ 
when $|x| \ls {\lambda}^{-{\delta}} \ll {\lambda}^{-\varepsilon}$,
which proves the result.
\end{proof}

We also have to estimate the term $ r_0(t,x) =  {\lambda}\im r(t,x, \partial_x \omega(t,x)) $
which in fact is bounded according to the following lemma.

\begin{lem}\label{transpterm}
If $\varepsilon = 1/8  $ and $ \delta = (1 + 2 \varepsilon)/3 = 5/12$ then $r_0(t,x) \in S(1, g_{\delta}) $ for  $|x| \ls \lambda^{-\delta}  $ and $ t \in I $.
\end{lem}

Observe that we need that $ \varepsilon < 1/7 $ and $ \delta = (1 + 2 \varepsilon)/3 $ in order to use the expansion of Proposition~\ref{expprop}, and when $\varepsilon = 1/8  $ we get $ \delta = 5/12 = 10\varepsilon/3 < 7\varepsilon/2$.

\begin{proof}
As before  we shall  use scaling $(t, x, {\xi}) = ({\lambda}^{-3\varepsilon}s,  {\lambda}^{-3\varepsilon}y , 
{\lambda}^{-4\varepsilon}{\eta})$, and write  $ f(s,y,\eta) = {\lambda}^{7\varepsilon}r(t,x,\xi_0 + \xi)  \in C^\infty $  and $\omega_0(s,y) = \lambda^{7\varepsilon}\omega(t,x)   \in C^\infty$   uniformly so that $ \partial_y \omega_0(s,y) =\lambda^{4\varepsilon} \partial_x \omega(t, x) $ when $ |x | \le c {\lambda}^{-3\varepsilon}$ and $ t \in I $, which we shall assume in the following. 
 
This gives 
\begin{equation}\label{error}
r_0(t,x) = {\lambda}^{1 -7\varepsilon} \im f(s,y, \partial_y \omega_0(t,y))
\end{equation}
and we shall show that 
\begin{equation*}
F(s,y) = \im f(s, y, \partial_y \omega_0(t, y)) \in S(\lambda^{-\varepsilon}, g_{\varrho}) \qquad \text{when $ |y| \le c \lambda^{-\varrho}$} 
\end{equation*}
where $ \varrho = \delta - 3 \varepsilon = \varepsilon/3$. Since $\varepsilon =  1/8  $, this will give the result.  Taylor's formula gives
\begin{multline}\label{Fexp}
F(s,y) = \partial_y \im f(s,0,0)y  + \w{\partial_y^2 \im f(s,0,0)y,y}/2  \\ + \partial_\eta \im f(s,0,0) \w{\partial_y^3\omega_0(s ,0)y,y}/2 +  R(s ,y) \qquad  |y| \le c \lambda^{-\varrho} 
\end{multline}
where $ R(s,y) \in C^\infty $ uniformly and vanishes of order $ 3 $ at $ y=0 $
since  $ f(s, 0,0) = \partial_y \omega_0(s,0) =  \partial_y^2\omega_0(s,0) = 0$ when $ t \in I $ by Propositions~\ref{symplprop} and~\ref{omegalem}.
Thus 
$$ 
R(s ,y) = \Cal O( |y|^3) = \Cal O(\lambda^{-3\varrho}) = \Cal O(\lambda^{-\varepsilon}) \quad \text{ when $ |y| \le c \lambda^{-\varrho} $ }
$$  
since $ \varrho =  \varepsilon/3 $. Now one loses at most a factor $y =  \Cal O(\lambda^{-\varrho})  =  \Cal O(\lambda^{-\varepsilon/3}) $ when taking a derivative of $R(s,y)$, giving a factor $  \Cal O(\lambda^{\varrho}) $, so $ R(s,y) \in S({\lambda}^{-\varepsilon}, g_{\varrho})$. 

It remains to consider the first three terms in~\eqref{Fexp} and as before it suffices to consider derivatives of order less than $ 3 $ at~$ y = 0 $.
Since $\partial_y^3\omega_0(s ,0) \in C^\infty$ uniformly we only have to estimate
$ \partial_\eta \im f(s,0,0)$ 
and $ \partial_{s,y}^k  \im f(s,0,0)$  when $ k \le 2 $. We obtain from~\eqref{cond001} that
$$ 
\partial_\eta \im f(s,0,0) = \lambda^{3\varepsilon} \partial_\xi \im r(t,0,\xi_0) =   \Cal O(\lambda^{-5\varepsilon/3})  =   \Cal O(\lambda^{-\varepsilon + 2\varrho})  
$$ 
Similarly, \eqref{cond001} gives
$$ \partial_{s,y} \im f(s,0,0) = \lambda^{4\varepsilon} \partial_{t,x} \im r(t,0,\xi_0)  =   \Cal O(\lambda^{-2\varepsilon/3})  =   \Cal O(\lambda^{-\varepsilon + \varrho})
$$ 
and~\eqref{cond002} gives that $\partial_{s,y}^2 \im f(s,0,0)  = \lambda^{\varepsilon} \partial_{t,x}^2 \im r (t,0,\xi_0)  = \Cal O(\lambda^{-\varepsilon/3})  =   \Cal O(\lambda^{-\varepsilon + 2 \varrho})$. 
\end{proof}

By a change of $ t $ variable we may assume that~\eqref{c3} and~\eqref{mcond2}
hold with the integration starting at $ t=0 $.
We obtain new variables $z$ in $\br^{n-1}$ by solving
\begin{equation*}
 \partial_t z_j = \w{a_j(t), z} \qquad z_j(0) = x_j \qquad \forall\, j
\end{equation*}
Then $D_t + \sum_{j}^{} \w{a_j(t), x} D_{x_j}$ is transformed into
$D_t$ but $D_{x_j} = D_{z_j}$ is unchanged, and we will for simplicity keep the notation $(t,x)$. 
The linear change of variables is uniformly bounded since $a_j \in C^\infty$, so
it preserves the neighborhoods $|x| \ls {\lambda}^{-{\nu}}$ and the
symbol classes $S({\lambda}^{\mu}, g_{\nu})$, $\forall\, {\mu},\, {\nu}$.
We shall then solve the approximate transport equation
\begin{equation}\label{mtrans}
 D_t{\varphi} + ({q_0}(t) + i r_0(t,x)){\varphi} = 0
\end{equation}
where ${\varphi}(0,x) \in S(1, g_{\delta})$ is supported where $|x| \ls
{\lambda}^{-{\delta}}$, $q_0(t)$ is given by~\eqref{q0def} and $ r_0 $ by~\eqref{key}.  
If we assume $ 3{\varepsilon} \le \delta \le 7\varepsilon/2$ then by
Lemma~\ref{translemma} the approximation errors $ R\varphi$ will be in 
$S({\lambda}^{3{\varepsilon}-{\delta}}, g_{{\delta}})$. In fact, since
$\partial_x$ maps $ S(1, g_{\delta})$ into $S({\lambda}^{{\delta}},
g_{\delta})$  we find $R(t,x,D_x){\varphi}_0
\in  S({\lambda}^{3{\varepsilon} - {\delta}}, g_{\delta})$  when $|x| \ls {\lambda}^{-{\delta}}$. 
We find from Proposition~\ref{wtpest} that $q_0 \in  S(\lambda^{\varepsilon} , g_{\varepsilon})$, and 
if $ \varepsilon =1/8 $ and $ \delta = (1 + 2 \varepsilon)/3 $ then we find from Lemma~\ref{transpterm} that $ r_0 \in S(1, g_\delta) $ when  $|x| \ls \lambda^{-\delta}  $ and $t \in I$.

If we choose the initial data
${\varphi}(0,x) = {\phi}_0(x)= {\phi}({\lambda}^{{\delta}}x)$,
where ${\phi} 
\in C^\infty_0$ satisfies ${\phi}(0) = 1$, we obtain the solution
\begin{equation}\label{vp}
 {\varphi}(t,x) =
 {\phi}_0(x) \exp(- i B(t,x)) 
\end{equation}
where $\partial_tB(t,x) = q_0(t) + i r_0(t,x)$ and $ B(0,x) = 0 $. 
We find that $ \exp(- i B(t,x))  \in S(1, g_{\delta})  $ uniformly
since condition~\eqref{c3} holds with $a_j \equiv 1$, $ \partial_{t} B(t,x) = q_0(t)  + i r_0(t,x) \in S(\lambda^{\varepsilon} , g_{\varepsilon}) + S(1, g_{\delta}) \subset S(\lambda^{\delta} , g_{\delta})$ and 
$$
\partial_{x} B(t,x) =  i\int_{0}^{t}   \partial_{x}r_0(s,x) \, ds \in S(\lambda^{\delta} , g_{\delta})
$$
by Proposition~\ref{wtpest} and Lemma~\ref{transpterm}. 
Thus $ \varphi \in S(1, g_\delta)$ uniformly and we find by~\eqref{vp}  that 
$ |{\varphi}(t,x)| \le  C|{\phi}({\lambda}^{{\delta}}x)|$ so $|x| \ls
{\lambda}^{-{\delta}}$ in $\supp {\varphi}$, which also holds in 
the original $x$ coordinates.

After solving the eikonal equation and the approximate transport
equation, we find from Proposition~\ref{expprop} that the terms in the
expansion~\eqref{exp} 
are $ \Cal
O({\lambda}^{3{\varepsilon}-\delta})$ if  ${\varepsilon} < 1/7$
and ${\delta} = (1+2{\varepsilon})/3 $,  
and all the terms
contain the factor $ \exp(- i B(t,x))$. We take ${\varepsilon}  =
1/8$ and ${\delta} = 5/12$ which gives
$ 3\varepsilon - \delta = -1/24 > -\varepsilon/2$ so $3 \varepsilon < \delta < 7\varepsilon/2 $. Then the expansion
in Proposition~\ref{expprop} is in
multiples of ~${\lambda}^{-1/24}$, and since the error terms 
of~\eqref{exp1} are $\Cal O({\lambda}^{-1/24})$
we will take ${\varrho} = 1/24$ and $ \varphi_0 = \varphi $
in the definition of $ u_\lambda $ given by~\eqref{udef}.

The approximate transport equation for ${\varphi}_k$ in~\eqref{udef}, $ k > 0$, is
\begin{equation}\label{gentrans}
D_t{\varphi}_k + ({q_0}(t) + i r_0(t,x)){\varphi}_k = {\lambda}^{k/24} R_k
\exp(i B(t,x))  \qquad k \ge 1
\end{equation}
with $R_k$   is
uniformly bounded in the symbol class $S({\lambda}^{-k/24},
g_{5/12})$ and  is  supported where $|x| \ls  {\lambda}^{-{5/12}}$.
In fact, $R_k$ contains the error terms from the transport
equation~\eqref{trans}
and also the terms that are $\Cal O({\lambda}^{-k/24})$ in \eqref{exp} depending on~${\varphi}_{j}$ for $j < k $.
Taking
 ${\varphi}_k =  \exp(-i B(t,x)) {\phi}_k$ we obtain the equation 
\begin{equation}\label{transk}
 D_t{\phi}_k = {\lambda}^{k/24} R_k \in S(1,
 g_{5/12})
\end{equation}
with initial values ${\phi}_k(0,x) = 0$,
which can be solved with ${\phi}_k \in 
 S(1, g_{5/12})$ uniformly having support where  $|x| \ls
{\lambda}^{-{5/12}}$. Since $ \exp(-i B(t,x)) \in 
S(1, g_{5/12}) $ uniformly we find that ${\varphi}_k \in 
S(1, g_{5/12})$ uniformly having support where  $|x| \ls
{\lambda}^{-{5/12}}$.
Proceeding by induction we obtain a solution to~\eqref{exp} modulo $\Cal
O({\lambda}^{-N/24})$ for any~ $N$.

\begin{prop}\label{transprop}
Assuming Propositions~\ref{prepprop} and ~\ref{symplprop} and
choosing ${\varepsilon} = 1/8$, ${\delta} = 5/12$  and ${\varrho} =
1/24$ we can solve the transport equations~\eqref{mtrans}
and~\eqref{gentrans} with ${\varphi}_k \in S(1,g_{5/12})$ 
having support where  $|x| \ls {\lambda}^{-{5/12}}$,
such that ${\varphi}_0(0,0) = 1$ and ${\varphi}_k(0,x)
\equiv 0$, $k \ge 1$. 
\end{prop}

Now, we get localization in $x$ from the initial values and the
transport equation. To get
localization in $t$ we use that $\im B(t) \le C$ so that $\re (-i B) \le C$. Near $\partial {\Gamma}$ we may assume that
$\re (-iB(t))  \ll -\log {\lambda}$ in an interval of length
$\Cal O({\lambda^{-\varepsilon}}) = \Cal O({\lambda}^{-1/8})$
by~\eqref{mcond2}. 
Thus by applying a cut-off function 
$
{\chi}(t) \in S(1, {\lambda}^{1/4}dt^2) \subset S(1, g_{5/12})
$ 
such that ${\chi}(0) = 1$ and ${\chi}'(t)$ is supported
where~\eqref{mcond2}  holds, i.e., where 
${\varphi}_k = \Cal O({\lambda}^{-N})$, $\forall\, k$, we obtain a solution modulo $\Cal
O({\lambda}^{-N})$ for any $N$. In fact, if $u_{\lambda}$ is defined
by~\eqref{udef} and $Q$ by Proposition~\ref{prepprop} then
$ Q {\chi} u_{\lambda} = {\chi}Q u_{\lambda} + [Q, {\chi}]u_{\lambda} $
where $[Q, {\chi}] = D_t {\chi}$ is supported
where $u_{\lambda} = \Cal O({\lambda}^{-N})$ which gives terms that are $\Cal
O({\lambda}^{-N})$, $\forall\, N$. Thus, by solving the eikonal
equation~\eqref{eic} for ${\omega}$ and the transport
equations~\eqref{gentrans} for ${\varphi}_k$ for $k \le 24 N$, we obtain that $Q{\chi}
u_{\lambda} = \Cal O({\lambda}^{-N})$ for any~$N$ and we get the
following remark.

\begin{rem}\label{transrem}
In Proposition~\ref{transprop} we may assume that ${\varphi}_k(t,x) =
{\phi}_k({\lambda}^{5/12}t, 
{\lambda}^{5/12}x) \in S(1, g_{5/12})$, $k \ge 0$, with ${\phi}_k
\in C_0^\infty$ having
support where $|x| \ls 1$ and $|t| \ls 
{\lambda}^{5/12}$, $k \ge 0$.
\end{rem}

\section{The proof of Theorem \ref{mainthm}}\label{pfsect}  

For the proof we will need the following modification
of~\cite[Lemma 26.4.14]{ho:yellow} which is Lemma~{7.1} in ~\cite{de:lim}. Recall that $\Cal D'_{{\Gamma}} =
\set{u \in \Cal D': \wf (u) \subset {\Gamma}}$ for ${\Gamma} \subset
T^*\br^n$, and that
$\mn{u}_{(k)}$ is the $L^2$ Sobolev norm of order $k$ of $u \in
C_0^\infty$.

\begin{lem}\label{estlem}
Let 
\begin{equation}\label{estlem0}
 u_{\lambda}(x) =  {\lambda}^{(n-1){\delta}/2}\exp(i{\lambda}^{\varrho}
 {\omega}({\lambda}^{\varepsilon}x)) 
 \sum_{j=0}^M 
 {\varphi}_{j} ({\lambda}^{{\delta}}x){\lambda}^{-j{\kappa}}
\end{equation}
with ${\omega} \in C^\infty (\br^n)$ satisfying $\im {\omega} \ge 0$
and $|d {\omega}| \ge c > 0$, ${\varphi}_j \in C^\infty_0(\br^n)$,
${\lambda} \ge 1$, ${\varepsilon}$, ${\delta}$, ${\kappa}$ and
${\varrho}$ are positive such that $\varepsilon < {\delta} < {\varepsilon} +
{\varrho}$. Here ${\omega}$ and 
${\varphi}_j$ may depend on ${\lambda}$ but uniformly, and ${\varphi}_j$ has fixed
compact support in all but one of the 
variables, for which the support is bounded by $C{\lambda}^{{\delta}}$.  
Then for any integer $N$ we have 
\begin{equation}\label{estlem1}
 \mn{u_{\lambda}}_{(-N)} \le C {\lambda}^{-N({{\varepsilon} + {\varrho}})}
\end{equation}
If ${\varphi}_0(x_0) \ne 0$ and $\im {\omega}(x_0) = 0$ for some $x_0$ then
there exists $c > 0$ and ${\lambda}_0 \ge 1$ so that
\begin{equation}\label{estlem2}
  \mn{u_{\lambda}}_{(-N)} \ge c {\lambda}^{-(N+
    \frac{n}{2})({\varepsilon}+{\varrho}) + (n-1){\delta}/2} \qquad
  {\lambda} \ge {\lambda}_0
\end{equation}
Let ${\Sigma} = \bigcap_{{\lambda} \ge 1} \bigcup_j  \supp
{\varphi}_j({\lambda}\, \cdot)$ 
and let $ {\Gamma}$ be the cone generated by 
\begin{equation}\label{estlem3}
 \set{(x,\partial{\omega}(x)),\ x
   \in {\Sigma},\ \im {\omega}(x) = 0} 
\end{equation}
then for any real $m$ we find ${\lambda}^m u_{\lambda} \to 0$ in $\Cal
D'_{ {\Gamma}}$ so 
${\lambda}^m Au_{\lambda} \to 0$ in $C^\infty$ if $A$ is a 
pseudodifferential operator such that $\wf(A) \cap  {\Gamma} =
\emptyset$. The estimates are uniform if
${\omega} \in C^\infty$ uniformly with fixed lower bound on $|d\re {\omega}|$, and
${\varphi}_j \in C^\infty$ uniformly.
\end{lem}

We shall use Lemma~\ref{estlem} for $u_{\lambda}$ in~\eqref{udef}, then
${\omega}$ will be real valued and
${\Gamma}$ in~\eqref{estlem3} will be the bicharacteristic
${\Gamma}_j$ converging to a limit bicharacteristic.

\begin{proof}[Proof of Lemma \ref{estlem}]
We shall adapt the proof of~\cite[Lemma
26.4.14]{ho:yellow} to this case. By making the change of variables $y
= {\lambda}^{\varepsilon}x$ we find that 
\begin{equation}\label{utrans}
 \hat u_{\lambda}({\xi}) = {\lambda}^{(n-1){\delta}/2- n{\varepsilon}} \sum_{j=0}^{M}
 {\lambda}^{-j{\kappa}} \int
 e^{i({\lambda}^{\varrho}{\omega}(y) 
 - \w{y,{\xi}/{\lambda}^{\varepsilon}})} {\varphi}_{j}({\lambda}^{{\delta} -
 {\varepsilon}}y)\,dy  
\end{equation}
Let $U$ be a neighborhood of the projection on the second
component of the set in~\eqref{estlem3}. When
${\xi}/{\lambda}^{\varepsilon +{\varrho}} \notin
U$ then for  ${\lambda} \gg 1$ we have that
\begin{multline*}
\bigcup_j \supp {\varphi}_j({\lambda}^{\delta - \varepsilon}\cdot) \ni y\mapsto ({\lambda}^{\varrho}{\omega}(y)  -
\w{y,{\xi}/{\lambda}^{\varepsilon}})/({\lambda}^{{\varrho}} +
 |{\xi}|/{\lambda}^{\varepsilon}) \\ = ({\omega}(y)  -
\w{y,{\xi}/{\lambda}^{\varepsilon +{\varrho}}})/(1 + |{\xi}|/
{\lambda}^{{\varepsilon} +{\varrho}})
\end{multline*}
is in  
a compact set of functions with non-negative imaginary part with a fixed
lower bound on the gradient of the real part. Thus, by integrating by
part in~\eqref{utrans} we find for any positive integer $m$ that 
\begin{equation}\label{pfest}
 |\hat u_{\lambda}({\xi})| \le C_m{\lambda}^{-(n-1)\delta /2 +
   m({\delta} - {\varepsilon})}( {\lambda}^{{\varrho}} +
 |{\xi}|/{\lambda}^{\varepsilon})^{-m} \qquad
 {\xi}/{\lambda}^{\varepsilon +{\varrho}} \notin U \qquad {\lambda} \gg 1
\end{equation}
This gives any negative power of ${\lambda}$ for $m$ large enough
since ${\delta} < {\varepsilon} + {\varrho}$.
If $V$ is bounded and $0 \notin \ol V$ then since $u_{\lambda}$ is
uniformly bounded in $L^2$ we find
\begin{equation*}
 \int_{{\tau}V}  |\hat u_{\lambda}({\xi})|^2 (1  +
 |{\xi}|^2)^{-N}\,d{\xi} \le C_V{\tau}^{-2N} \qquad {\tau} \ge 1
\end{equation*}
Using this estimate with ${\tau} = {\lambda}^{{\varepsilon} + {\varrho}}$
together with the estimate~\eqref{pfest} we obtain~\eqref{estlem1}.  
If ${\chi}
\in C_0^\infty$ then we may apply~\eqref{pfest} to
${\chi}u_{\lambda}$, thus we find for any positive integer $j$ that
\begin{equation*}
  |\widehat {{\chi}u}_{\lambda}({\xi})| \le C_j{\lambda}^{-(n-1)\delta /2 +
   j({\delta} - {\varepsilon})}( {\lambda}^{{\varrho}} +
 |{\xi}|/{\lambda}^{\varepsilon})^{-j} \qquad {\xi}
  \in W \qquad {\lambda} \gg 1
\end{equation*}
if $W$ is any closed cone with $
{\Gamma} \bigcap (\supp {\chi}\times W) = \emptyset$. 
Thus we find that
${\lambda}^m u_{\lambda} \to 0$ in $\Cal D'_{ {\Gamma}}$ for every $m$.
To prove \eqref{estlem2} we assume $x_0 = 0$ and take ${\psi}\in
C_0^\infty$. If $\im {\omega}(0) = 0$ and ${\varphi}(0) \ne
0$ we find
\begin{multline*}
 {\lambda}^{n({\varepsilon} + {\varrho}) - (n-1){\delta}/2} 
 e^{-i{\lambda}^{\varrho}\re {\omega}(0)}\w{u_{\lambda},
   {\psi}({\lambda}^{{\varepsilon} + \varrho}\cdot)}\\ = \int 
 e^{i{\lambda}^{\varrho}({\omega}(x/{\lambda}^{\varrho}) - {\omega}(0))}{\psi}(x) 
 \sum_{j}{\varphi}_j(x/{\lambda}^{{\varepsilon} + \varrho - {\delta}})
 {\lambda}^{-j{\kappa}}\,dx  \\\to \int
 e^{i\w{\re \partial_x{\omega}(0),x}}{\psi}(x)
 {\varphi}_0(0)\,dx \qquad {\lambda} \to + \infty
\end{multline*}
which is not equal to zero for some suitable ${\psi} \in
C^\infty_0$. In fact, we have 
${\varphi}_j(x/{\lambda}^{{\varepsilon} + \varrho - {\delta}}) =
{\varphi}_j(0) + \Cal O({\lambda}^{{\delta} -
  {\varepsilon}-{\varrho}}) \to {\varphi}_j(0)$ when ${\lambda} \to \infty$, because
${\delta} < {\varepsilon} + {\varrho}$. Since
\begin{equation*}
 \mn{{\psi}({\lambda}^{{\varepsilon}+{\varrho}} \cdot)}_{(N)} \le C
 {\lambda}^{(N-n/2)({{\varepsilon}+{\varrho}})} 
\end{equation*}
we obtain that $0 < c \le  {\lambda}^{(N+
\frac{n}{2})({\varepsilon}+{\varrho}) - (n-1){\delta}/2}\mn{u}_{(-N)}$ which  
gives~\eqref{estlem2} and the lemma. 
\end{proof}

\begin{proof}[Proof of Theorem~\ref{mainthm}]
Assume that ${\Gamma}$ is a limit bicharacteristic of $P$.
We are going to show that~\eqref{solvest} does not hold for any
${\nu}$, $N$ and any  
pseudodifferential operator $A$ such that ${\Gamma} \cap \wf (A) =
\emptyset$. This means that there exists
  $0 \ne u_j \in C^\infty_0$
such that  
\begin{equation}\label{solvest0}
 \mn {u_j}_{(-N)}/(\mn{P^*{u_j}}_{({\nu})} + \mn {u_j}_{(-N-n)} +
 \mn{Au_j}_{(0)}) \to \infty  \qquad \text{when $ j \to \infty$}
\end{equation}
which will contradict the local solvability of $P$ at~${\Gamma}$ by
Remark~\ref{solvrem}. 

Let ${\Gamma}_j \subset \Sigma \bigcap S^*X $ be a sequence of semibicharacteristics of $p$\/ that converges to the limit bicharacteristic ${\Gamma} \subset \st$ and let $ \lambda_j $ be given by~\eqref{cond1} and~\eqref{ldef} with $ \varepsilon > 0$ which will be chosen later. 
Now the conditions and conclusions are
invariant under symplectic changes of homogeneous coordinates and
multiplication by elliptic pseudodifferential operators. By
Proposition~\ref{prepprop} we may assume
that the coordinates are chosen so that ${\Gamma}_j = I \times
(0,0,{\xi}_j)$ with $|{\xi}_j| = 1$, and for any $0 < {\varepsilon} < 1/3$
and $ c > 0 $
we can write $B_j P^* = Q_j + R_j \in \Psi^{1-\varepsilon}_{1-{\varepsilon},{\varepsilon}}$  where $B_j  \in 
{\Psi}^\varepsilon_{1-{\varepsilon},{\varepsilon}}$ uniformly,  ${\Gamma}_j 
\cap \wf_{\varepsilon} (R) = \emptyset$ uniformly and $Q_j$ has
symbol  
\begin{equation}\label{normform}
{\tau} - r(t,x,{\xi}) + q_0(t,x,{\xi}) + r_0(t,x,{\xi})
\end{equation}
when the homogeneous distance to 
${\Gamma}_j$ is less than $c|\xi|^{-\varepsilon} \ls \lambda_j^{-\varepsilon} $. We have that $r_0 \in
S^{3{\varepsilon} - 1}_{1-{\varepsilon},{\varepsilon}}$, $q_0 \in
S^{\varepsilon}_{1-{\varepsilon},{\varepsilon}}$ is
given by~\eqref{symb}, and $r \in  
S^{1-{\varepsilon}}_{1-{\varepsilon},{\varepsilon}}$ with real part vanishing of
second order at~${\Gamma}_j$, and the bounds are uniform in the symbol classes.

Now, we may replace the norms
$\mn{u}_{(s)}$ in~\eqref{solvest0} by the norms
\begin{equation*}
 \mn u_s^2 = \mn{\w{D_x}^s u}^2 = \int \w{{\xi}}^{2s}|\hat
 u({\tau},{\xi})|^2\, d{\tau}d{\xi}
\end{equation*}
and the corresponding spaces $H_s$.
In fact, the quotient $\w{{\xi}}/\w{({\tau},{\xi})} \cong 1$ 
when $|{\tau}| \ls |{\xi}|$, thus in a conical neighborhood of
${\Gamma}$. So replacing the norms  
in the estimate~\eqref{solvest0} only changes the
constant and the operator $A$ in the estimate~\eqref{solvest}.
By using Proposition~\ref{symplprop} we may assume that
the grazing Lagrangean space $L_j(w) \equiv \set{(s,y;0,0): \ (s,y) \in
\br^{n}}$, $\forall\, w \in {\Gamma}_j$, after conjugation with a 
uniformly bounded $ C^1 $ section $ F(t) $ of homogeneous Fourier integral operators,
then $\partial_x^2 \re r = 0$ at ${\Gamma}_j$.
Observe that for each ~$ t $ we find that
$ F(t) $ is uniformly continuous in local  $ H_s $ spaces, which we may use in~\eqref{solvest0} after changing~$ A $.
Also the conjugation of $ F(t) $ with the operator with symbol~\eqref{normform}
has a uniformly bounded expansion. 
In fact, this follows since $t \mapsto F(t) \in C^1$ are  homogeneous Fourier integral operators in the $x$ variables and these preserve the symbol classes. By changing $ A $ again, we may then replace the local $\mn{u}_{s}$
norms by the norms $\mn{u}_{(s)}$ in~\eqref{solvest0} so that we can use Lemma~\ref{estlem}.

Now, by choosing  ${\delta} = 5/12$, ${\varepsilon} = 1/8$ and ${\varrho}
= 1/24$ and using
Propositions~\ref{expprop}, \ref{omegalem}, ~\ref{transprop} and
Remark~\ref{transrem}, we can for each ${\Gamma}_j$ construct   
approximate solution $u_{{\lambda}_j}$ on the form~\eqref{udef} 
so that $Q
u_{{\lambda}_j} = \Cal O({\lambda}^{-k}_j)$, for any~$k$. The real valued
phase function is equal to
$\w{x,{\xi}_j} + {\omega}_j(t,x)$ where $ |{\xi}_j| = 1 $ and ${\omega}_j(t,x) \in
S({\lambda_j}^{-7/8}, g_{3/8})$ and the values of 
$$
(t,x; {\lambda}_j\partial_t{\omega}_j(t,x),{\lambda}_j({\xi}_j +
\partial_x{\omega}_j(t,x)))
$$ 
have homogeneous distance  $ \ls \lambda_j^{-5/12} $
to the rays through  ${\Gamma}_j$ when $|x| \ls
{\lambda_j}^{-5/12}$, thus on $\supp u_{{\lambda}_j}$. 
Observe that if $ \lambda_j \gg 1 $ then we have that $ |\xi_0 + \partial_x\omega_j(t,x)| \cong 1 $ 
in  $\supp u_{{\lambda}_j}$. In fact, we have 
 $$
{\omega}_j(t,x) = {\lambda_j}^{-7/8} \wt
{\omega}_j({\lambda_j}^{3/8}t,{\lambda_j}^{3/8}x)
 $$ 
where $\wt {\omega}_j
\in C^\infty$ uniformly so $\partial_x {\omega}_j = \Cal
O({\lambda}_j^{-1/2})$. Now
$$
{\lambda_j}(\w{x,{\xi}_j} + {\omega}_j(t,x)) =
{\lambda_j}^{5/8}\w{{\lambda_j}^{3/8}x,{\xi}_j} + {\lambda_j}^{1/8}\wt
{\omega}_j({\lambda_j}^{3/8}t,{{\lambda_j}^{3/8}x})\quad \text{when $ |x|  \ls {\lambda}_j^{-5/12}$}
$$ 
thus ${\delta}= 5/12$, ${\varrho} = 5/8$,
${\varepsilon} = 3/8$ and  ${\kappa} = 1/24$
in~\eqref{estlem0} so ${\varepsilon} + {\varrho} = 
1 > {\delta} > \varepsilon$.

The amplitude functions for $u_{\lambda_j} $ are  ${\varphi}_{k,j}(t,x) =
{\phi}_{k,j}({\lambda_j}^{5/12}t, {\lambda_j}^{5/12}x)$ where
${\phi}_{k,j} \in 
C^\infty_0$ uniformly in~$j$ with fixed compact support in~$x$, but
in~$t$ the support is bounded by
$C{\lambda_j}^{5/12}$. Thus $u_{{\lambda}_j}$
will satisfy the conditions in Lemma~\ref{estlem} uniformly.
Clearly differentiation of~$Qu_{\lambda_j}$ can at most give a factor
${\lambda}_j$ since ${\delta} < {\varepsilon} + {\varrho} = 1$. Because of the bound on the
support of $u_{\lambda_j}$ we may obtain that 
\begin{equation} \label{8.13}
\mn{Q u_{\lambda_j}}_{({\nu})} = \Cal
O({\lambda}_j^{-N-n})
\end{equation} 
for any given  ${\nu}$. 

If $\wf (A) \bigcap {\Gamma} = \emptyset$,
then we find $\wf (A) \bigcap {\Gamma}_j = \emptyset$ for large $j$,
so Lemma~\ref{estlem} gives $\mn{Au_{{\lambda}_j}}_{(0)} 
= \Cal O({\lambda}_j^{-N-n})$ when $ j \to \infty $.
On $\supp u_{{\lambda}_j}$ we have $x = \Cal O({\lambda}_j^{-5/12})$  so
the values of $(t,x; {\lambda}_j\partial_t{\omega}_j(t,x), {\lambda}_j({\xi}_j +
\partial_x{\omega}_j(t,x)))$  have homogeneous distance  $ \ls  \lambda_j^{-5/12}$ to the rays through ${\Gamma}_j$.
Thus, if $R_j \in S^{9/8}_{7/8,1/8}$ such that
$\wf_{1/8}(R_j) \bigcup {\Gamma}_j = \emptyset$ uniformly then we find from the
expansion~\eqref{trevesexp}  
that all the terms of~$R_ju_{{\lambda}_j}$ vanish for large enough
~${\lambda}_j$. In fact,  
since ${\lambda}_j^{-5/12} \ll {\lambda}_j^{-1/8}$ for $j \gg 1$, we find 
for any~${\alpha} $ and~$K$ that
$$
\partial^{\alpha} R_j(t,x;{\lambda}_j((0,\xi_j) + \partial_{t,x} {\omega}_j(t,x))) =
O({\lambda}_j^{-K})
$$ 
in $\bigcup_k \supp {\varphi}_{k,j}$. As before, we find that
$\mn{R_ju_{{\lambda}_j}}_{(\nu)} = \Cal O({\lambda}_j^{-N-n})$
by the bound on the support of~$u_{\lambda_j}$, so we obtain
from~\eqref{8.13} that  
\begin{equation} 
\mn{P^*u_{\lambda_j}}_{({\nu})} = \Cal
O({\lambda}_j^{-N-n})
\end{equation} 
for  any given ${\nu}$. 

Since  ${\varepsilon} + {\varrho} =
1$  and $ \delta >  0 $ we also find from  Lemma~\ref{estlem} that
\begin{equation*}
 {\lambda}_j^{-N} = {\lambda}_j^{-N({\varepsilon}+ {\varrho})}\gs 
\mn{u_{{\lambda}_j}}_{(-N)}  \gs 
 {\lambda}_{j} ^{-(N+
    \frac{n}{2})({\varepsilon}+{\varrho}) + (n-1){\delta}/2}  \ge {\lambda}_j^{-N-n/2}
\end{equation*}
when ${\lambda}_j \ge 1$. We obtain that~\eqref{solvest0} holds for
$u_j = u_{{\lambda}_j}$  when $j \to \infty$, 
so Remark~\ref{solvrem} gives that~$P$ is not solvable at the limit
bicharacteristic~${\Gamma}$. 
\end{proof}

\bibliographystyle{plain}
\bibliography{nec}

\end{document}